\documentclass[12pt]{article}
\usepackage{amsthm}
\usepackage{amssymb}
\usepackage{amsfonts}
\usepackage{mathrsfs}
\usepackage[all,cmtip]{xy}
\usepackage{tikz}
\usepackage{tikz-cd}
\usepackage{multirow}
\usepackage{url}
\usepackage{graphicx}
\usepackage{epstopdf}
\usepackage{mathtools}
\usepackage{leftidx}
\usepackage{enumerate}
\usepackage{enumitem}
\usepackage{wrapfig}

\usetikzlibrary{cd}

\usepackage{indentfirst}
\numberwithin{equation}{section}

\DeclareFontFamily{U}{cal}{}
\DeclareFontShape{U}{cal}{m}{n}{<->cmsy10}{}

\DeclareSymbolFont{rcal}{U}{cal}{m}{n}
\DeclareSymbolFontAlphabet{\mathcal}{rcal}

\newtheorem{Def}{Definition}[section]
\newtheorem{Bsp}[Def]{Example}
\newtheorem{Prop}[Def]{Proposition}
\newtheorem{Theo}[Def]{Theorem}
\newtheorem{Lem}[Def]{Lemma}
\newtheorem{Koro}[Def]{Corollary}
\theoremstyle{definition}
\newtheorem{Rem}[Def]{Remark}

\newcommand{\bsm}{\begin{smallmatrix}}
	\newcommand{\esm}{\end{smallmatrix}}

\newcommand{\add}{{\rm add}}
\newcommand{\gd}{{\rm gl.dim }}

\newcommand{\End}{{\rm End}}

\newcommand{\rd} {{\rm rep.dim}}

\def\wrd{\mathop{\rm w.resol.dim}\nolimits}

\def\rd{\mathop{\rm resol.dim}\nolimits}
\def\rep{\mathop{\rm rep.dim}\nolimits}

\newcommand{\rad}{{\rm rad}}

\renewcommand{\top}{{\rm top}}
\newcommand{\pd}{{\rm pd}}
\newcommand{\id}{{\rm id}}

\newcommand{\I}{{\mathcal I}}

\newcommand{\T}{{\mathcal T}}

\newcommand{\V}{{\mathcal V}}

\newcommand{\cpx}[1]{#1^{\bullet}}
\newcommand{\D}[1]{{\mathscr D}(#1)}

\newcommand{\Db}[1]{{\mathscr D}^b(#1)}
\newcommand{\C}[1]{{\mathscr C}(#1)}

\newcommand{\K}[1]{{\mathscr K}(#1)}

\newcommand{\Kb}[1]{{\mathscr K}^b(#1)}
\newcommand{\modcat}{\ensuremath{\text{{\rm -mod}}}}

\newcommand{\Modcat}{\ensuremath{\text{{\rm -Mod}}}}

\newcommand{\stmodcat}[1]{#1\mbox{{\rm -{\underline{mod}}}}}

\newcommand{\pmodcat}[1]{#1\mbox{{\rm -proj}}}
\newcommand{\imodcat}[1]{#1\mbox{{\rm -inj}}}

\newcommand{\opp}{^{\rm op}}

\newcommand{\Hom}{{\rm Hom}}

\newcommand{\lra}{\longrightarrow}
\newcommand{\lraf}[1]{\stackrel{#1}{\lra}}
\newcommand{\ra}{\rightarrow}

\newcommand{\ITdist}{{\rm IT.dist}}
\newcommand{\Kfb}[1]{{\mathscr K}^{-,b}(#1)}
\newcommand{\Kpb}[1]{{\mathscr K}^{+,b}(#1)}
\newcommand{\per}{\mathsf {per}}
\newcommand{\Dsg}{\mathscr{D}_{\mathrm{sg}}}

\title{ \bf  Igusa-Todorov distances
	\\
	\footnotetext{Email addresses: zhangjb@ahu.edu.cn, zhengjunling@cjlu.edu.cn}
}

\author {Jinbi Zhang, Junling Zheng
	\thanks{Corresponding author} \\
	{\it \scriptsize  School of Mathematical Sciences, Anhui University, Hefei, 230601, Anhui Province, PR China}\\
	{\it \scriptsize  Department of Mathematics, China Jiliang University, Hangzhou, 310018, Zhejiang Province, PR China}
}
\date{}
\begin{document}
	
	\maketitle
	\begin{abstract}
		A new homological dimension, called the Igusa-Todorov distance, is introduced to measure how far an Artin algebra is from being an Igusa-Todorov algebra.
		An upper bound for the dimension is established in terms of the Loewy length, leading to the conclusion that every Artin algebra has a finite Igusa-Todorov distance.
		Using this dimension, we derive an upper bound for the dimension of the singularity category. 
		Furthermore, we investigate how the Igusa–Todorov distance behaves under various relationships between algebras. 
		Specifically, we demonstrate that stable equivalences preserve the Igusa-Todorov distances for algebras without nodes, prove that it is an invariant under singular equivalence of Morita type with level, and establish bounds for the distances of algebras involved in a recollement of derived module categories. Consequently, the Igusa-Todorov distance is an invariant under derived equivalences of algebras.
	\end{abstract}
	
	{\footnotesize\tableofcontents\label{contents}}	
	
	\section{Introduction}
	
	Let $A$ be an Artin algebra. We denote by $A\modcat$ the category of finitely generated left $A$-modules and by $\Db{A\modcat}$ the bounded derived category of $A\modcat$. Recall that the \textit{finitistic dimension} of $A$ is defined to be the supermum of the projective dimensions of all finitely projective modules of finite projective dimension. The famous finitistic dimension conjecture states that the finitisctic dimension of an Artin algebra is always finite (see \cite[conjecture (11), p. 410]{ars97}). The conjecture is still open now. It is closely related to several conjectures in the representation theory (see \cite[pp. 409-410]{ars97}).
	In order to study the conjecture, Igusa and Todorov introduced two functions, $\Phi$ and $\Psi$, now known as the Igusa-Todorov functions (\cite{it05} ).These functions have since become an important tool in the study of the finitistic dimension conjecture (see \cite{wei09,x06}). Using the properties of Igusa-Todorov functions, Wei defines $n$-Igusa-Todorov algebras (\cite{wei09}), which satisfy the finitistic dimension conjecture. A natural question arises, namely, are all Artin algebras Igusa-Todorov algebras?
	However, Conde showed that not every Artin algebra is an Igusa-Todorov algebra (\cite{conde16}). 
	In order to measure how far an algebra is from an Igusa-Todorov algebra, we introduce the notation of the Igusa-Todorov distance for Artin algebras. 
	
	When introducing a new homological dimension, a basic question arises: Does every Artin algebra have a finite Igusa–Todorov distance?
	In this paper, we give a positive answer to this question by proving an upper bound for the Igusa–Todorov distance.
	
	\begin{Theo}\label{main-theo-bound}
		{\rm (Theorem \ref{theo-bound})}
		Let $A$ be an Artin algebra and $\ell\ell(A)$ its Loewy length.
		Then $$\ITdist(A)\le \max\{\ell\ell(A)-2,0\}.$$
	\end{Theo}
	
	The dimension of triangulated categories, introduced by Rouquier, is an invariant that measures how efficiently a triangulated category can be generated from a single object.
	This notion has proven useful in studying the representation dimension of Artin algebras (see \cite{0pp09, rou06, rou08}).
	The singularity category $\Dsg(A)$ of an algebra $A$ is defined as the Verdier quotient of the bounded derived category $\Db{A\modcat}$ by the full subcategory of perfect complexes (see \cite{Buch86, Orl04}). 
	The singularity category captures the homological singularity of the algebra: $A$ has finite global dimension if and only if its singularity category is trivial.
	In this paper, we establish an upper bound for the dimension of the singularity category in terms of the Igusa–Todorov distance.
	
	\begin{Theo}\label{main-sing-it}
		{\rm (Theorem \ref{sing-it})}
		Let $A$ be an Artin algebra. Then $$\dim(\Dsg(A))\le \ITdist(A).$$
	\end{Theo}
	
	Note that the precise value of the Igusa–Todorov distance for a given algebra is generally very difficult to compute directly. A possible approach is to investigate how this distance behaves under certain well-structured relationships between algebras.
	
	In representation theory, a fundamental relation is stable equivalence. 
	Mart\'inez-Villa showed that stable equivalences preserve both the global and dominant dimensions of algebras without nodes (\cite{MV1990}).
	Guo showed that the representation dimension is also preserved under stable equivalences (\cite{Guo05}); this result had already been established by Xi in the case of stable equivalences of Morita type (\cite{Xi02} ). 
	Koenig and Liu demonstrated that simple-minded systems are invariant under stable equivalences (\cite{kl12}). 
	Xi and Zhang showed that the delooping levels,  $\phi$-dimensions and $\psi$-dimensions of Artin algebras are preserved under stable equivalences for algebras without nodes (\cite{Xi2022}).
	Recently, Zhang and Zheng showed that extension dimensions are stable equivalence invariants for Artin algebras (\cite{zz24}). 
	In this paper, we prove that the Igusa–Todorov distance is also preserved under stable equivalences between Artin algebras without nodes.
	
	\begin{Theo}\label{main-st-thm}
		{\rm (Theorem \ref{st-thm})}
		Let $A$ and $B$ be stably equivalent Artin algebras without nodes. Then $$\ITdist(A)=\ITdist(B).$$
	\end{Theo}
	
	Two algebras are said to be {\it singularly equivalent} if their singularity categories are equivalent as triangulated categories. In particular, derived equivalent algebras are singularly equivalent, but the converse is not true in general. For this reason, many scholars are devoted to extending the properties preserved under derived equivalences to singular equivalences (see \cite{Chen18,chqw23,Ska16, Wang21}). A particularly important class of such equivalences arises from bimodules, as introduced by Chen and Sun (\cite{CS12}) and generalized by Wang to singular equivalences of Morita type with level (\cite{Wang15}).
	Such equivalences have been shown to preserve various homological conjectures. For instance, the finitistic dimension conjecture and Keller’s conjecture on singular Hochschild cohomology are invariant under singular equivalences of Morita type with level (\cite{CLW20, Wang15}). We contribute the following:
	
	\begin{Theo}\label{main-thm-level}
		{\rm (Theorem \ref{thm-level})}
		Let $A$ and $B$ be finite dimensional algebras.
		If $A$ and $B$ are singularly equivalent of Morita type with level, then $$\ITdist(A)=\ITdist(B).$$
	\end{Theo}
	
	Recollements of triangulated categories have been introduced by
	Beilinson, Bernstein and Deligne (\cite{BBD}) in order to decompose
	a derived category of constructible sheaves into two parts, an open and a closed one. Recollements of derived categories are powerful tools in connecting the homological properties of three algebras $A$, $B$ and $C$. They have been used to study properties such as global or finitistic dimension \cite{akly17,cx17,happel93}, K-theory \cite{akly17,cx12,sch06} and Hochschild (co)homology \cite{han14,Keller98,kn09}, self-injective dimension and $\Phi$-dimension \cite{qin18}, syzygy-finite properties and Igusa-Todorov properties \cite{ww22}.
	We analyze the Igusa–Todorov distance in the context of recollements:
	
	\begin{Theo}\label{main-thm-rec}
		{\rm (Theorem \ref{thm-rec})}
		Let $A$, $B$ and $C$ be three Artin algebras.
		Suppose that there is a recollement among the derived categories
		$\D{A\Modcat}$, $\D{B\Modcat}$ and $\D{C\Modcat}$$:$
		\begin{align}\label{main-rec}
			\xymatrixcolsep{4pc}\xymatrix{
				\D{B\Modcat} \ar[r]|{i_*=i_!} &\D{A\Modcat} \ar@<-2ex>[l]|{i^*} \ar@<2ex>[l]|{i^!} \ar[r]|{j^!=j^*}  &\D{C\Modcat}. \ar@<-2ex>[l]|{j_!} \ar@<2ex>[l]|{j_{*}}
			}
		\end{align}
		
		{\rm (1)(i)} If the recollement $(\ref{main-rec})$ extends one step downwards, then $\ITdist(C)\le \ITdist(A)$.
		
		{\rm (ii)} If the recollement $(\ref{main-rec})$ extends one step upwards, then $\ITdist(B)\le \ITdist(A)$.
		
		{\rm (2)} If the recollement $(\ref{main-rec})$ extends one step downwards and extends one step upwards, then $$\max\{\ITdist(B),\ITdist(C)\}\le \ITdist(A)\le \ITdist(B)+\ITdist(C)+1.$$
		
		{\rm (3)(i)} If the recollement $(\ref{main-rec})$ extends one step downwards and $\gd(B)<\infty$, then $$\ITdist(C)= \ITdist(A).$$
		
		{\rm (ii)} If $\gd(C)<\infty$, then $\ITdist(B)= \ITdist(A)$.
	\end{Theo}
	
	As we known, derived equivalences play an important role in the representation theory of Artin algebras and finite groups (see \cite{happel88, Xi18}). 
	Foundational results such as Rickard’s Morita theory for derived categories of rings (\cite{r1989}) and Keller’s work on differential graded algebras (\cite{k94}) have significantly advanced our understanding of the homological behavior shared by derived equivalent algebras.
	A number of homological invariants are known to be preserved under derived equivalences, including Hochschild homology (\cite{r91}), cyclic homology (\cite{Keller98}), algebraic $K$-theory (\cite{ds04}), and the number of non-isomorphic simple modules (\cite{r1989}). In many cases, derived equivalences impose upper bounds on homological dimensions via the length of the tilting complex realizing the equivalence. This behavior is observed, for instance, in the global dimension (\cite[Section 12.5(b)]{gr92}), the finitistic dimension (\cite{happel93, px09}), the extension dimension (\cite[Theorem 1.1]{zz24}), and the dominant dimension (\cite[Theorem 5.2]{fhk21}).
	However, such dimensions are not necessarily preserved exactly.
	In this work, we demonstrate that the Igusa–Todorov distance remains invariant under derived equivalences.
	
	\begin{Koro}{\rm (Corollary \ref{cor-der})}
		\label{main-cor-der}
		Let $A$ and $B$ be Artin algebras. If $A$ and $B$ are derived-equivalent, then $$\ITdist(A)=\ITdist(B).$$
	\end{Koro}
	
	The paper is organized as follows:
	In Section 2, we recall the necessary definitions, notations, and background.
	Section \ref{sect-def} introduces the Igusa–Todorov distance and establishes its upper bound.
	Section \ref{sect-st} proves Theorem \ref{main-st-thm} on stable equivalences.
	Section \ref{sect-t-level} establishes Theorem \ref{main-thm-level} on singular equivalences of Morita type with level.
	Section \ref{sect-t-rec} proves Theorem \ref{main-thm-rec} on recollements.
	Corollary \ref{main-cor-der} then follows as a consequence.
	
	\section{Preliminaries}
	In this section, we shall fix some notations, and recall some definitions.
	
	Throughout this paper, $k$ is an arbitrary but fixed commutative Artin ring. Unless stated otherwise, all algebras are Artin $k$-algebras with unit, and all modules are finitely generated unitary left modules; all categories will be $k$-categories and all functors are $k$-functors.
	
	Let $A$ be an Artin algebra and $\ell\ell(A)$ stand for the Loewy length of $A$. We denote by $A\Modcat$ the category of all left $A$-modules, by $A\modcat$ the category of all finitely generated left $A$-modules, by $\pmodcat{A}$ the category of all finitely generated projective $A$-modules, and by $\imodcat{A}$ the category of all finitely generated projective $A$-modules. All subcategories of $A\Modcat$ are full, additive and closed under isomorphisms.
	For a class $\mathcal{C}$ of $A$-modules, we write $\add(\mathcal{C})$ for the smallest full subcategory of $A\modcat$ containing $\mathcal{C}$ and closed under finite direct sums and direct summands.
	When $\mathcal{C}$ consists of only one module $C$, we write $\add(C)$ for $\add(\mathcal{C})$.
	In particular, $\add({_A}A)=\pmodcat{A}$.
	Let $M$ be an $A$-module.
	If $f:P\ra M$ is the projective cover of $M$ with $P$ projective, then the kernel of $f$ is called the \emph{syzygy} of $M$, denoted by $\Omega(M)$.
	Dually, if $g:M\ra I$ is the injective envelope of $M$ with $I$ injective, then the cokernel of $g$ is called the \emph{cosyzygy} of $M$, denoted by $\Omega^{-1}(M)$. Additionally, let $\Omega^0$ be the identity functor in $A\modcat$ and $\Omega^1:=\Omega$. Inductively, for any $n\ge 2$, define $\Omega^n(M):=\Omega^1(\Omega^{n-1}(M))$ and $\Omega^{-n}(M):=\Omega^{-1}(\Omega^{-n+1}(M))$.
	We denoted by $\pd(_AM)$ and $\id(_AM)$ the projective and injective dimension, respectively.
	
	Let $A^{\rm {op}}$ be the opposite algebra of $A$, and $A^e = A\otimes _k A^{\opp}$ be the enveloping algebra of $A$. We identify
	$A$-$A$-bimodules with left $A^e$-modules. Let $D:=\Hom_k(-,E(k/\rad(k)))$ the usual duality from $A\modcat$ to $A^{\rm {op}}\modcat$, where $\rad(k)$ denotes the radical of $k$ and $E(k/\rad(k))$ denotes the injective envelope of $k/\rad(k)$. The duality $\Hom_A(-, A)$ from $\pmodcat{A}$ to $\pmodcat{A\opp}$ is denoted by $^*$, namely for each projective $A$-module $P$, the projective $A\opp$-module $\Hom_A(P, A)$ is written as $P^*$. We write $\nu_A$ for the Nakayama functor $D\Hom_A(-,A): \pmodcat{A}\ra \imodcat{A}$.

	\medskip
	Let $\cal C$ be an additive category. For two morphisms $f:X\rightarrow Y$ and $g:Y\rightarrow Z$ in $\cal C$, their composition is denoted by $fg$, which is a morphism from $X$ to $Z$. But for two functors $F:\mathcal{C}\ra \mathcal{D}$ and $G:\mathcal{D}\ra\mathcal{E}$ of categories, their composition is written as $GF$.
	
	Suppose $\mathcal{C}\subseteq A\Modcat$. A (co)chain complex $\cpx{X}=(X^i, d_X^i)$ over $\mathcal{C}$ is a sequence of objects $X^i$ in
	$\cal C$ with morphisms $d_{\cpx{X}}^{i}:X^i\ra X^{i+1}$ such that $d_{\cpx{X}}^{i}d_{\cpx{X}}^{i+1}=0$ for all $i \in {\mathbb Z}$.
	A (co)chain map $f$ from $\cpx{X}=(X^i, d_X^i)$ to $\cpx{Y}=(Y^i, d_Y^i)$, is a set of maps $f=\{f^i:X^i\ra Y^i\mid i\in \mathbb{Z}\}$ such that $f^id^i_{\cpx{Y}}=d^i_{\cpx{X}}f^{i+1}$. For a complex $\cpx{X}$, the complex $\cpx{X}[1]$ is obtained from $\cpx{X}$ by shifting $\cpx{X}$ to the left by one degree, and the complex $\cpx{X}[-1]$ is obtained from $\cpx{X}$ by shifting $\cpx{X}$ to the right by one degree. For $n\in \mathbb{Z}$, we denoted by $\cpx{X}_{\le n}$ (respectively, $\cpx{X}_{\ge n}$) the brutal trucated complex which is obtained from the given complex $\cpx{X}$ by replacing each $X^i$ with $0$ for $i>n$ (respectively, $i<n$).
	A complex $\cpx{X}=(X^i, d_X^i)$
	is called \textit{bounded above} (respectively, \textit{bounded below}) if $X^{i} = 0$ for all but finitely many positive (respectively, negative) integers $i$. A complex $\cpx{X}$ is called \textit{bounded} if it is both bounded above and bounded below, equivalently,
	$X^{i} = 0$ for all but finitely many $i$. A complex $\cpx{X}$ is called \textit{cohomologically bounded} if all but finitely many cohomologies of $\cpx{X}$ are zero.
	
	We denote by $\C{\mathcal C}$  the category of all complexes over $\cal C$ with chain map, by $\K{\mathcal C}$ the homotopy category of complexes over $\mathcal{C}$, by $\mathscr{K}^{-,b}(\mathcal C)$ the homotopy category of bounded above complexes over $\mathcal{C}$, by $\mathscr{K}^{+,b}(\mathcal C)$ the homotopy category of bounded below complexes over $\mathcal{C}$, and by $\mathscr{K}^{b}(\mathcal C)$ the homotopy category of bounded complexes over $\mathcal{C}$. Let $I$ be a subset of $\mathbb{Z}$. We denoted by $\mathscr{K}^{I}(\mathcal C)$ the subcategory of $\mathscr{K}(\mathcal C)$ consisting of complexes with the $i$-th component is $0$ for each $i\notin I$. For instance, $\mathscr{K}^{\{0\}}(\mathcal C)=\mathcal{C}$.
	If $\cal C$ is an abelian category, then let $\D{\cal C}$ be the derived category of complexes over $\cal C$, $\Db{\mathcal C}$ be the full subcategory of $\D{\cal C}$ consisting of cohomologically bounded complexes over $\mathcal{C}$, and $\mathscr{D}^{I}(\mathcal C)$ be the subcategory of $\D{\mathcal C}$ consisting of complexes with the $i$-th cohomology is $0$ for any $i\in I$. For instance, $\mathscr{D}^{\{0\}}(\mathcal C)=\mathcal{C}$.
	
	Let $A$ be an Artin algebra. For convenience, we do not distinguish $\Kb{\pmodcat{A}}$, $\Kb{\imodcat{A}}$ and $\Db{A\modcat}$ from their essential images under the canonical full embeddings into $\D{A\Modcat}$. Furthermore, we always identify $A\Modcat$ with the full subcategory of $\D{A\Modcat}$ consisting of all stalk complexes concentrated on degree $0$.
	
	In this paper, all functors between triangulated categories are assumed to be triangle functors.
	
	\subsection{Syzygies in derived categories}
	In this section, we recall some basic and self-contained facts on syzygy complexes for later use, as detailed in reference \cite{ai07,wei17}.
	
	Let $A$ be an Artin algebra.
	A homomorphism $f: X\ra Y$ of $A$-modules is called a \emph{radical homomorphism} if, for any indecomposable module $Z$ and homomorphisms $h: Z\ra X$ and $g: Y\ra Z$, the composition $hfg$ is not an isomorphism.
	For a complex $(X^i, d_{\cpx{X}}^i)$ over $A\modcat$, if all $d_{\cpx{X}}^i$ are radical homomorphisms, then it is called a \emph{radical complex}, which has the following properties.
	\begin{Lem}\label{lem-radical}
		{\rm (\cite[pp. 112-113]{hx10})}
		Let $A$ be an Artin algebra.
		
		$(1)$ Every complex over $A\modcat$ is isomorphic to a radical complex in $\K{A\modcat}$.
		
		$(2)$ Two radical complexes $\cpx{X}$ and $\cpx{Y}$ are isomorphic in $\K{A\modcat}$ if and only if they are isomorphic in $\C{A\modcat}$.
	\end{Lem}
	
	Recall that $\Db{A\modcat}$ is equivalent to $\Kfb{\pmodcat{A}}$ as triangulated categories. For a complex $\cpx{X}\in \Db{A\modcat}$, a \textit{minimal projective resolution} of $\cpx{X}$ is a radical complex $\cpx{P}\in\Kfb{\pmodcat{A}}$ such that $\cpx{P}\simeq \cpx{X}$ in $\D{A\modcat}$. In case $X$ is an $A$-module, $\cpx{P}$ is just the usual minimal projective resolution of $X$.
	
	\begin{Def}{\rm (\cite[Section 1.3]{ai07})}
		{\rm
			Let $\cpx{X}\in \Db{A\modcat}$ and $n\in \mathbb{Z}$. Let $\cpx{P}$ be a minimal projective resolution of $\cpx{X}$. We say that a complex in $\Db{A\modcat}$ is an $n$\textit{-th syzygy} of $\cpx{X}$ provied that it is isomorphic to $\cpx{P}_{\le -n}[-n]$ in $\D{A\modcat}$. In this case, the $n$-th syzygy of $\cpx{X}$ is denoted by $\Omega_{\mathscr{D}}^n(\cpx{X}_{\cpx{P}})$, or simply by $\Omega_{\mathscr{D}}^n(\cpx{X})$ if there is no danger of confusion.}
	\end{Def}
	\begin{Rem}
		Given a complex $\cpx{X}\in \Db{A\modcat}$, the definition of the $n$-th syzygy of $\cpx{X}$  we adopt here differs slightly from Wei's definition. We take the projective resolution of $\cpx{X}$ to be minimal. According to Lemma \ref{lem-radical}, we know that the $n$-th syzygy of $\cpx{X}$ is unique uo to isomorphism.
	\end{Rem}

	\begin{Bsp}{\rm (\cite[Example 3.2]{wei17})}{\rm
			Let $\cpx{X}$ be a complex in $\Db{A\modcat}$. Suppose that $\cpx{P}$ is a minimal projective resolution of $\cpx{X}$. Assume that $\cpx{P}$ is a complex of the form
			$$\cdots \lra P^{-(n+1)}\lra P^{-n}\lra \cdots\lra P^{-1}\lra P^0\lra P^{1}\lra \cdots \lra P^{m-1}\lra P^{m}\lra 0,$$
			where $P^0$ is at the $0$-th position and $n,m>0$. Then
			
			{\rm (1)} $\Omega_{\mathscr{D}}^n(\cpx{X})$ is a complex of the form
			$$\cdots \lra P^{-(n+1)}\lra P^{-n}\lra 0,$$
			where $P^{-n}$ is at the $0$-th position.
			
			{\rm (2)} $\Omega_{\mathscr{D}}^{-1}(\cpx{X})$ is a complex of the form
			$$\cdots \lra P^{-(n+1)}\lra P^{-n}\lra \cdots\lra P^{-1}\lra P^0\lra P^{1}\lra 0,$$
			where $P^{1}$ is at the $0$-th position.
			
			{\rm (3)} $\Omega_{\mathscr{D}}^{-m}(\cpx{X})$ is a complex of the form
			$$\cdots \lra P^{-(n+1)}\lra P^{-n}\lra \cdots\lra P^{-1}\lra P^0\lra P^{1}\lra \cdots \lra P^{m-1}\lra P^{m}\lra 0,$$
			where $P^{m}$ is at the $0$-th position.
			
			{\rm (4)} $\Omega_{\mathscr{D}}^{-(m+1)}(\cpx{X})$ is a complex of the form
			$$\cdots \lra P^{-(n+1)}\lra P^{-n} \lra \cdots \lra P^0\lra P^{1}\lra \cdots \lra P^{m-1}\lra P^{m}\lra 0\lra 0,$$
			where the first $0$ after $P^{m}$ is at the $0$-th position.
		}
	\end{Bsp}
	
	Let $X$ be an $A$-module and $\cpx{P}$ be the minimal projective resolution of $X$. Then the brutal truncated complex $\cpx{P}_{\le -n}[-n]$ is just the minimal projective resolution of the $n$-th syzygy of $X$ for $n>0$. Thus syzygies of $X$ coincide with the usual syzygies in module categories.
	
	The following list some basic properties of syzygy complexes in \cite{wei17}.
	
	\begin{Lem}\label{lem-prop-omega}
		{\rm (\cite[Lemma 3.3]{wei17})}
		Let $\cpx{X}, \cpx{Y}\in\Db{A\modcat}$, and $m,n,s,t$ be integers. Let $\cpx{P}$ be the minimal projective resolutions of $\cpx{X}$. Then
		
		{\rm (1)} $\Omega_{\mathscr{D}}^n(\cpx{X})\in \mathscr{D}^{(-\infty,0]}(A\modcat)$ and $\Hom_{\Db{A\modcat}}(Q,\Omega_{\mathscr{D}}^n(\cpx{X})[i])=0$ for any projective $A$-module $Q$ and any integer $i>0$.
		
		{\rm (2)}
		If $\cpx{X}\in \mathscr{D}^{[s,t]}(A\modcat)$ for some integers $s\le t$, then $\Omega_{\mathscr{D}}^n(\cpx{X})\in \mathscr{D}^{[0,0]}(A\modcat),$ $\mathscr{D}^{[s+n,0]}(A\modcat),$ $\mathscr{D}^{[s+n,t+n]}(A\modcat)$ for cases $n\ge -s$, $-t\le n\le -s$, $n\le -t$, respectively. In particular, $\Omega_{\mathscr{D}}^{-s}(\cpx{X})$ is isomorphic to an $A$-module and $\cpx{X}\simeq \Omega_{\mathscr{D}}^{-t}(\cpx{X})[-t]$.
		
		{\rm (3)} $\Omega_{\mathscr{D}}^{n+m}(\cpx{X}[m])\simeq \Omega_{\mathscr{D}}^n(\cpx{X})$.
		
		{\rm (4)} $\Omega_{\mathscr{D}}^{n+m}(\cpx{X})\simeq \Omega_{\mathscr{D}}^m(\Omega_{\mathscr{D}}^n(\cpx{X}))$ for $m\ge 0$.
		
		%
		{\rm (5)}
		$\cpx{X}\in \Kb{\pmodcat{A}}$ if and only if any/some syzygy of $\cpx{X}$ is also in $\Kb{\pmodcat{A}}$.
		
		{\rm (6)} $\Omega_{\mathscr{D}}^n(\cpx{X}\oplus \cpx{Y})\simeq \Omega_{\mathscr{D}}^n(\cpx{X})\oplus \Omega_{\mathscr{D}}^n(\cpx{Y})$.
	\end{Lem}
	
	\begin{Lem}\label{lem-x-tri}
		{\rm (\cite[Lemma 3.4]{wei17})}
		Let $\cpx{X}\in \Db{A\modcat}$. Then there is a triangle $\Omega_{\mathscr{D}}^{n+1}(\cpx{X})[n-m]\ra \cpx{Y}\ra \Omega_{\mathscr{D}}^m(\cpx{X})\ra \Omega_{\mathscr{D}}^{n+1}(\cpx{X})[n-m+1],$ where $\cpx{Y}\in \mathscr{K}^{[m-n,0]}(\pmodcat{A})$ and $n\ge m$. In particular, for each $n\in \mathbb{Z}$, there is a triangle $\Omega_{\mathscr{D}}^{n+1}(\cpx{X})\ra Q\ra \Omega_{\mathscr{D}}^n(\cpx{X})\ra \Omega_{\mathscr{D}}^{n+1}(\cpx{X})[1]$ with $Q$ projective.
	\end{Lem}
	
	\begin{Lem}\label{lem-omega-shift}
		{\rm (\cite[Proposition 3.8]{wei17})} Let $\cpx{X}\ra \cpx{Y}\ra \cpx{Z}\ra \cpx{X}[1]$ be a triangle in $\Db{A\modcat}$, and $n$ be an integer. Then there is a trianle $\Omega_{\mathscr{D}}^n(\cpx{X})\ra \Omega_{\mathscr{D}}^n(\cpx{Y})\oplus P\ra \Omega_{\mathscr{D}}^n(\cpx{Z})\ra \Omega_{\mathscr{D}}^n(\cpx{X})[1]$ for some projective $A$-module $P$.
	\end{Lem}
	
	Recall that the dual notation of syzygy complexes, that is , cosyzygy complexes. Note that $\Db{A\modcat}$ is equivalent to $\Kpb{\imodcat{A}}$ as triangulated categories. For a complex $\cpx{X}\in \Db{A\modcat}$, a injective resolution of $\cpx{X}$ is a complex $\cpx{I}\in\Kpb{\imodcat{A}}$ such that $\cpx{I}\simeq \cpx{X}$ in $\D{A\modcat}$. In case $X$ is an $A$-module, $\cpx{I}$ is just a usual injective resolution of $X$.
	
	\begin{Def}{\rm (\cite[Definition 3.1$'$]{wei17})}
		{\rm
			Let $\cpx{X}\in \Db{A\modcat}$ and $n\in \mathbb{Z}$. Let $\cpx{I}$ be a injective resolution of $\cpx{X}$. We say that a complex in $\Db{A\modcat}$ is an $n$\textit{-th cosyzygy} of $\cpx{X}$ provied that it is isomorphic to $\cpx{I}_{\ge n}[n]$ in $\D{A\modcat}$. In this case, the $n$-th cosyzygy of $\cpx{X}$ is denoted by $\Omega^{\mathscr{D}}_n(\cpx{X}_{\cpx{I}})$, or simply by $\Omega^{\mathscr{D}}_n(\cpx{X})$ if there is no danger of confusion.
	}\end{Def}

	We state the dual results in  \cite[Section 3]{wei17}.
	
	\begin{Lem}
		{\rm (\cite[Lemma 3.3$'$]{wei17})}
		Let $\cpx{X}, \cpx{Y}\in\Db{A\modcat}$, and $m,n,s,t$ be integers. Let $\cpx{I}$ be a injective resolutions of $\cpx{X}$. Then
		
		{\rm (1)} $\Omega^{\mathscr{D}}_n(\cpx{X})\in \mathscr{D}^{[0,\infty)}(A\modcat)$ and $\Hom_{\Db{A\modcat}}(\Omega^{\mathscr{D}}_n(\cpx{X}),J[i])=0$ for any injective $A$-module $J$ and any integer $i>0$.
		
		{\rm (2)}
		If $\cpx{X}\in \mathscr{D}^{[s,t]}(A\modcat)$ for some integers $s\le t$, then $\Omega^{\mathscr{D}}_n(\cpx{X})\in \mathscr{D}^{[0,0]}(A\modcat),$
		$\mathscr{D}^{[0,t-n]}(A\modcat),$
		$\mathscr{D}^{[s-n,t-n]}(A\modcat)$
		for cases $n\ge t$, $s\le n\le t$, $n\le s$, respectively. In particular, $\Omega^{\mathscr{D}}_{t}(\cpx{X})$ is isomorphic to an $A$-module and $\cpx{X}\simeq \Omega_{\mathscr{D}}^{s}(\cpx{X})[s]$.
		
		{\rm (3)} $\Omega^{\mathscr{D}}_{n+m}(\cpx{X}[m])\simeq \Omega^{\mathscr{D}}_n(\cpx{X})$.
		
		{\rm (4)} $\Omega^{\mathscr{D}}_{n+m}(\cpx{X})\simeq \Omega^{\mathscr{D}}_m(\Omega^{\mathscr{D}}_n(\cpx{X}))$ for $m\ge 0$.
		
		{\rm (5)}
		$\cpx{X}\in \Kb{\imodcat{A}}$ if and only if any/some cosyzygy of $\cpx{X}$ is also in $\Kb{\imodcat{A}}$.
		
		{\rm (6)} $\Omega^{\mathscr{D}}_n(\cpx{X}\oplus \cpx{Y})\simeq \Omega^{\mathscr{D}}_n(\cpx{X})\oplus \Omega^{\mathscr{D}}_n(\cpx{Y})$.
	\end{Lem}
	
	\begin{Lem}{\rm (\cite[Lemma 3.5$'$]{wei17})}
		Let $\cpx{X}\in \Db{A\modcat}$. Then there is a triangle $\Omega^{\mathscr{D}}_{m}(\cpx{X})\ra\cpx{Y}\ra
		\Omega^{\mathscr{D}}_{n+1}(\cpx{X})[m-n]\ra \Omega^{\mathscr{D}}_{m}(\cpx{X})[1],$ where $\cpx{Y}\in \mathscr{K}^{[0,n-m]}(\imodcat{A})$ and $n\ge m$. In particular, for any $n$, there is a triangle $\Omega^{\mathscr{D}}_{n}(\cpx{X})\ra J\ra \Omega^{\mathscr{D}}_{n+1}(\cpx{X})\ra \Omega^{\mathscr{D}}_{n}(\cpx{X})[1]$ with $J$ injective.
	\end{Lem}
	
	\begin{Lem}{\rm (\cite[Proposition 3.8$'$]{wei17})} Let $\cpx{X}\ra \cpx{Y}\ra \cpx{Z}\ra \cpx{X}[1]$ be a triangle in $\Db{A\modcat}$, and $n$ be an integer. Then there is a trianle $\Omega^{\mathscr{D}}_n(\cpx{X})\ra \Omega^{\mathscr{D}}_n(\cpx{Y})\oplus I\ra \Omega^{\mathscr{D}}_n(\cpx{Z})\ra \Omega^{\mathscr{D}}_n(\cpx{X})[1]$ for some injective $A$-module $I$.
	\end{Lem}
	
	\subsection{Recollements}
	
	In this subsection, we recall the notation of recollements introduced by Beilinson, Bernstein and Deligne in \cite{BBD}, which is a very useful tool for representation theory and algebraic geometry.
	
	\begin{Def}{\rm
			Let $\mathscr{D}$, $\mathscr{D}'$ and $\mathscr{D}''$ be triangulated categories with shift functors denoted by $[1]$.
			
			{\rm (1)}
			{\rm (\cite{BBD})} A \emph{recollement} of $\mathscr{D}$ by $\mathscr{D}'$ and $\mathscr{D}''$ is a diagram of six triangle functors
			\begin{align}\label{diag:rec}
				\xymatrixcolsep{4pc}\xymatrix{\mathscr{D}' \ar[r]|{i_*=i_!} &\mathscr{D} \ar@<-2ex>[l]|{i^*} \ar@<2ex>[l]|{i^!} \ar[r]|{j^!=j^*}  &\mathscr{D}'', \ar@<-2ex>[l]|{j_!} \ar@<2ex>[l]|{j_{*}}
				}
			\end{align}
			which satisfies
			
			{\rm (R1)} $(i^\ast,i_\ast)$,\,$(i_!,i^!)$,\,$(j_!,j^!)$ ,\,$(j^\ast,j_\ast)$
			are adjoint pairs;
			
			{\rm (R2)}
			$i_*,~j_*,~j_!$ are fully faithful$;$
			
			{\rm (R3)}
			$j^*\circ i_*=0$ $($thus $i^*\circ j_!=0$ and $i^!\circ j_*=0$$);$
			
			{\rm (R4)}
			for any object $X$ in $\mathscr{D}$, there are two triangles in $\mathscr{D}$ induced by counit and unit adjunctions$:$
			$$ \xymatrix@R=0.5pc{i_*i^!(X)\ar[r] &X \ar[r] & j_*j^*(X) \ar[r]& i_*i^!(X)[1]\\
				j_!j^*(X) \ar[r] &X \ar[r] & i_*i^*(X) \ar[r]& j_!j^*(X)[1]} $$
			The upper two rows
			\begin{align*}
				\xymatrixcolsep{4pc}\xymatrix{\mathscr{D}' \ar[r]|{i_*} &\mathscr{D} \ar@<-2ex>[l]|{i^*}  \ar[r]|{j^*}  &\mathscr{D}'' \ar@<-2ex>[l]|{j_!}
				}
			\end{align*}
			is said to be a \emph{left recollemet} of $\mathscr{D}$ by $\mathscr{D}'$ and $\mathscr{D}''$ if the four functors $i^*$, $i_*$, $j_!$ and $j^!$ satisfy the conditions in {\rm (R1)-(R4)} involving them.  A \emph{right recollement} is defined similarly via the lower two rows.
			
			{\rm (2)}
			{\rm (\cite{akly17})}
			A \emph{ladder} of $\mathscr{D}$ by $\mathscr{D}'$ and $\mathscr{D}''$ is a finite or infinite diagram with triangle functors
			\vspace{-10pt}$${\setlength{\unitlength}{0.7pt}
				\begin{picture}(200,170)
					\put(0,70){$\xymatrix@!=3pc{\mathscr{D}'
							\ar@<+2.5ex>[r]|{i_{n-1}}\ar@<-2.5ex>[r]|{i_{n+1}} &\mathscr{D}
							\ar[l]|{j_n} \ar@<-5.0ex>[l]|{j_{n-2}} \ar@<+5.0ex>[l]|{j_{n+2}}
							\ar@<+2.5ex>[r]|{j_{n-1}} \ar@<-2.5ex>[r]|{j_{n+1}} &
							\mathscr{D}''\ar[l]|{i_n} \ar@<-5.0ex>[l]|{i_{n-2}}
							\ar@<+5.0ex>[l]|{i_{n+2}}}$}
					\put(52.5,10){$\vdots$}
					\put(137.5,10){$\vdots$}
					\put(52.5,130){$\vdots$}
					\put(137.5,130){$\vdots$}
			\end{picture}}$$
			such that any three consecutive rows form a recollement. The rows are labelled by a subset of $\mathbb{Z}$ and multiple occurence of the same recollement is allowed. The \emph{height} of a ladder is the number of recollements
			contained in it $($counted with multiplicities$).$ It is an element of $\mathbb{N}\cup \{0,\infty\}$. A recollement is viewed to be a ladder of height $1$.
	}\end{Def}
	
	Recall that a sequence of triangulated categories $\mathscr{D}'\stackrel{F}{\to}\mathscr{D}\stackrel{G}{\to}\mathscr{D}''$ is said to be a \emph{short exact sequence} (up to direct summands) if $F$ is fully faithful, $G\circ F=0$ and the induced functor $\overline{G}\colon \mathscr{D}/\mathscr{D}'\to\mathscr{D}''$ is an equivalence (up to direct summands). The following result is well-known, see \cite[1.4.4, 1.4.5, 1.4.8]{BBD} and \cite[Chapter III, Lemma 1.2 (1) and Chapter IV, Proposition 1.11]{BR07}.
	
	\begin{Prop}
		\label{prop:recollement-vs-ses}
		$(1)$ The two rows of a left recollement are short exact sequences of triangulated categories.
		Conversely, assume that there is a short exact sequence of triangulated categories {\rm (}possibly up to direct summands{\rm)}
		$$\xymatrixcolsep{2pc}\xymatrix{\mathscr{D}' \ar[r]^{i_*} &\mathscr{D}   \ar[r]^{j^*}  &\mathscr{D}''.
		}$$
		Then $i_*$ has a left adjoint {\rm (}respectively, right adjoint{\rm)} if and only if $j^*$ has a left adjoint {\rm (}respectively, right adjoint{\rm)}. In this case, $i_*$ and $j^*$ together with their left adjoints {\rm (}respectively, right adjoints{\rm)} form a left recollement {\rm (}respectively, right recollement{\rm)} of $\mathscr{D}$ in terms of $\mathscr{D}'$ and $\mathscr{D}''$.
		
		$(2)$ Assume that there is a diagram
		\begin{align*}
			\xymatrixcolsep{4pc}\xymatrix{\mathscr{D}' \ar[r]|{i_*=i_!} &\mathscr{D} \ar@<-2ex>[l]|{i^*} \ar@<2ex>[l]|{i^!} \ar[r]|{j^!=j^*}  &\mathscr{D}'' \ar@<-2ex>[l]|{j_!} \ar@<2ex>[l]|{j_{*}}
			}
		\end{align*} satisfying the condition {\rm (R1)}. If it is a recollement, then all the three rows are short exact sequences of triangulated categories. Conversely, if any one of the three rows is a short exact sequence of triangulated categories, then the diagram is a recollement.
		
		$(3)$ Given a recollement \eqref{diag:rec}, assume that $\mathscr{D}$, $\mathscr{D}'$ and $\mathscr{D}''$ admit small coproducts. Then both $j_!$ and $i^*$ preserve compact objects.
	\end{Prop}
	
	\begin{Def}{\rm
			We say that a recollement \eqref{diag:rec} \emph{extends one step downwards}  if both $i^!$ and $j_*$ have right adjoints. In this case, let $i_{?}$ and $j^{\#}$ be the right adjoints of $i^!$ and $j_*$, respectively. Then we have the following diagram:
			\begin{align}\label{ladder-2}
				\xymatrixcolsep{4pc}\xymatrix{\mathscr{D}' \ar[r]|{i_*=i_!} \ar@<-4ex>[r]|{i_{?}}
					&\mathscr{D} \ar@<-2ex>[l]|{i^*} \ar@<2ex>[l]|{i^!} \ar[r]|{j^!=j^*} \ar@<-4ex>[r]|{j^{\#}} &\mathscr{D}''. \ar@<-2ex>[l]|{j_!} \ar@<2ex>[l]|{j_{*}}
				}
			\end{align}
			By Proposition~\ref{prop:recollement-vs-ses}, the lower three rows also form a recollement. This means that (\ref{ladder-2}) form a ladder of height $2$. Similarly, we have the notion of extending the recollement one step upwards.
	}\end{Def}
	
	Now, we consider the recollements of derived module categories. Let $F:\Db{A\modcat}\to \Db{B\modcat}$ be a triangle functor. We say that $F$ \emph{restricts to} $\Kb{\rm proj}$ (respectively, $\Kb{\rm inj}$) if $F$ sends $\Kb{\pmodcat{A}}$ (respectively, $\Kb{\imodcat{A}}$) to $\Kb{\pmodcat{B}}$ (respectively, $\Kb{\imodcat{B}}$). Let $F:\D{A\Modcat}\to \D{B\Modcat}$ be a triangle functor. We say that $F$ \emph{restricts to} $\Db{\rm mod}$ (respectively, $\Kb{\rm proj}$, $\Kb{\rm inj}$) in a similar sense.
	
	The following result is well-known, see \cite{akly17,JYZ23}.
	
	\begin{Lem}\label{lem-rec-prop}
		Let $A$, $B$ and $C$ be Artin algebras.
		Suppose that there is a recollement among the derived categories
		$\D{A\Modcat}$, $\D{B\Modcat}$ and $\D{C\Modcat}$:
		\begin{align}\label{rec-der}
			\xymatrixcolsep{4pc}\xymatrix{
				\D{B\Modcat} \ar[r]|{i_*=i_!} &\D{A\Modcat} \ar@<-2ex>[l]|{i^*} \ar@<2ex>[l]|{i^!} \ar[r]|{j^!=j^*}  &\D{C\Modcat}. \ar@<-2ex>[l]|{j_!} \ar@<2ex>[l]|{j_{*}}
			}
		\end{align}
		Then $i^{*}$ and $j_{!}$ restrict to $\mathscr{K}^b({\rm proj})${\rm ;} $i_{*}$ and $j^{!}$ restrict to $\mathscr{D}^b({\rm mod})$. Moreover, the following hold true.
		
		{\rm (1)} The following conditions are equivalent$:$
		
		{\rm (i)} the recollement {\rm (\ref{rec-der})} extends one step downwards$;$
		
		{\rm (ii)} $i_{*}$ restricts to $\mathscr{K}^b({\rm proj})$$;$
		
		{\rm (iii)} $i_{*}(B)\in\Kb{\pmodcat{A}}$$;$
		
		{\rm (iv)} $j^{*}$ restricts to $\mathscr{K}^b({\rm proj})$$;$
		
		{\rm (v)} $j^{*}(A)$ restricts to $\mathscr{K}^b({\pmodcat{C}})$$;$
		
		{\rm (vi)} $i^!$ restricts to $\mathscr{D}^b({\rm mod})$$;$
		
		{\rm (vii)} $j_*$ restricts to $\mathscr{D}^b({\rm mod})$$;$
		
		{\rm (viii)} $i^!$ has a right adjoint$;$
		
		{\rm (ix)} $j_*$ has a right adjoint.
		
		In this case, the recollement {\rm (\ref{rec-der})} restricts to a left recollement and a right recollement
		\begin{align*}
			\xymatrixcolsep{4pc}\xymatrix{
				\Kb{\pmodcat{B}} \ar[r]|{i_*=i_!} &\Kb{\pmodcat{A}} \ar@<-2ex>[l]|{i^*}  \ar[r]|{j^!=j^*}  &\Kb{\pmodcat{C}}, \ar@<-2ex>[l]|{j_!}
			}
		\end{align*}
		\begin{align*}
			\xymatrixcolsep{4pc}\xymatrix{
				\Db{B\modcat} \ar[r]|{i_*=i_!} &\Db{A\modcat}  \ar@<2ex>[l]|{i^!} \ar[r]|{j^!=j^*}  &\Db{C\modcat}. \ar@<2ex>[l]|{j_{*}}
			}
		\end{align*}
		
		{\rm (2)} The following conditions are equivalent$:$
		
		{\rm (i)} the recollement {\rm (\ref{rec-der})} extends one step upwards$;$
		
		{\rm (ii)} $i_{*}$ restricts to $\mathscr{K}^b({\rm inj})$$;$
		
		{\rm (iii)} $i_{*}(D(B))\in \Kb{\imodcat{A}}$$;$
		
		{\rm (iv)} $j^{*}$ restricts to $\mathscr{K}^b({\rm inj})$$;$
		
		{\rm (v)} $j^{*}(D(A))\in \Kb{\imodcat{C}}$$;$
		
		{\rm (vi)} $i^*$ restricts to $\mathscr{D}^b({\rm mod})$$;$
		
		{\rm (vii)} $j_{!}$ restricts to $\mathscr{D}^b({\rm mod})$$;$
		
		{\rm (viii)} $i^*$ has a left adjoint$;$
		
		{\rm (ix)} $j_{!}$ has a left adjoint.
		
		In this case, the recollement {\rm (\ref{rec-der})} restricts to a left recollement and a right recollement
		\begin{align*}
			\xymatrixcolsep{4pc}\xymatrix{
				\Db{B\modcat} \ar[r]|{i_*=i_!} &\Db{A\modcat}  \ar@<-2ex>[l]|{i^*}  \ar[r]|{j^!=j^*}  &\Db{C\modcat}, \ar@<-2ex>[l]|{j_!}
			}
		\end{align*}
		\begin{align*}
			\xymatrixcolsep{4pc}\xymatrix{
				\Kb{\imodcat{B}}  \ar[r]|{i_*=i_!} &\Kb{\imodcat{A}} \ar@<2ex>[l]|{i^!} \ar[r]|{j^!=j^*}  &\Kb{\imodcat{C}}. \ar@<2ex>[l]|{j_{*}}
			}
		\end{align*}
	\end{Lem}
	
	\section{Definition of Igusa-Todorov distances}
	\label{sect-def}
	In this section, we introduce the notion of the Igusa-Todorov distance of an Artin algebra.
	
	\begin{Def}\label{def-sf-it}
		{\rm Let $A$ be an Artin algebra, and $n$ be a nonnegative integer.
			
			{\rm (1)} $A$ is said to be $n$-\emph{syzygy-finite} if $\Omega^n(A\modcat)$ is representation-finite, that is, the number of non-isomorphic indecomposable direct summands of modules in $\Omega^n(A\modcat)$ is finite. $A$ is said to be \emph{syzygy-finite} if $A$ is $n_0$-syzygy-finite for some $n_0$.
			
			{\rm (2)} {\rm (\cite[Definition 2.2]{wei09})}
			$A$ is said to be an $n$\textit{-Igusa-Todorov algebra} if there exists an $A$-module $U$ such that for any $A$-module $X$ there exists an exact sequence
			$$0\lra U_{1}\lra U_{0} \lra \Omega^{n}(X) \lra 0$$
			where $U_{i} \in \add(_AU)$ for each $0 \le i \le 1$.
			Such a module $U$ is said to be an $n$-Igusa-Todorov module.
			$A$ is said to be an \emph{Igusa-Todorov algebra} if $A$ is an $n_0$-Igusa-Todorov algebra for some $n_0$.
		}
	\end{Def}
	
	\begin{Def}\label{def-italg}
		{\rm (\cite[Definition 2.1]{zheng22})}
		{\rm
			Let $m$ and $n$ be nonnegative integers. An Artin algebra $A$ is said
			to be an $(m,n)$\textit{-Igusa-Todorov algebra} if there is a module
			$U\in A\modcat$ such that for any module $X\in A\modcat$ there exists an exact sequence
			$$0\lra U_m\lra U_{m-1}\lra \cdots \lra U_1\lra U_0\lra \Omega^n(X)\lra 0$$
			where $U_i \in \add(U)$ for each $0 \le i \le m $.
			Such a module $U$ is said to be an $(m,n)$-Igusa-Todorov module.
		}
	\end{Def}
	
	To measure how far an algebra is from being Igusa-Todorov algebras, 
	the Igusa-Todorov distance of an Artin algebra is defined now.
	\begin{Def}\label{def-itdist}
		{\rm 
			Let $A$ be an Artin algebra.
			We set the \textit{Igusa-Todorov distance} of $A$ as follows
			$$\ITdist(A):=\inf\{m \mid A \mbox{ is an } (m,n) \mbox{-Igusa-Todorov algebra for some } n\}.$$
		}
	\end{Def}
	
	\begin{Lem}\label{lem-it}
		Let $A$ be an Artin algebra.
		
		{\rm (1)} $\ITdist(A)=0$ if and only if $A$ is syzygy-finite.
		
		{\rm (2)} $\ITdist(A)\le 1$ if and only if $A$ is Igusa-Todorov algebra.
	\end{Lem}
	\begin{proof}
		This lemma follows from Definition \ref{def-itdist} and Definition \ref{def-sf-it}.
	\end{proof}
	At the end of this subsection, we provide an example that is not syzygy-finite.
	\begin{Bsp}{\rm
			{\rm (\cite[Example 54]{bm23})}
			Let $A =kQ/I$ be an algebra where $Q$ is
			$$
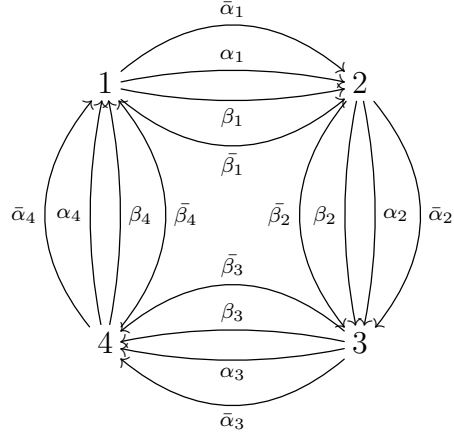
\xymatrix{ 1 \ar@/^8mm/[rrr]^{\bar{\alpha}_1} \ar@/^2mm/[rrr]^{\alpha_1} \ar@/_2mm/[rrr]_{\beta_1} \ar@/_8mm/[rrr]_{\bar{\beta_1}} & &  & 2 \ar@/^8mm/[ddd]^{\bar{\alpha}_2} \ar@/^2mm/[ddd]^{\alpha_2} \ar@/_2mm/[ddd]_{\beta_2} \ar@/_8mm/[ddd]_{\bar{\beta_2}} \\ & &  &\\& & & & \\ 4 \ar@/^8mm/[uuu]^{\bar{\alpha}_4} \ar@/^2mm/[uuu]^{\alpha_4} \ar@/_2mm/[uuu]_{\beta_4} \ar@/_8mm/[uuu]_{\bar{\beta_4}} &  & &  3 \ar@/^8mm/[lll]^{\bar{\alpha}_3} \ar@/^2mm/[lll]^{\alpha_3} \ar@/_2mm/[lll]_{\beta_3} \ar@/_8mm/[lll]_{\bar{\beta_3}}}$$
			and $I = \langle \alpha_{i}\alpha_{i+1}-\bar{\alpha}_{i}\bar{\alpha}_{i+1},\ \beta_{i}\beta_{i+1}-\bar{\beta}_{i}\bar{\beta}_{i+1},\ \alpha_{i}\bar{\alpha}_{i+1},\ \bar{\alpha}_{i}\alpha_{i+1},\ \beta_{i}\bar{\beta}_{i+1},\ \bar{\beta}_{i}\beta_{i+1},\text{ for } i \in \mathbb{Z}_4, \ J^3 \rangle$.
			It follows from \cite{bm23} that $\Omega^{n}(A\modcat)$ is infinite representation type for each $n\in\mathbb{N}$. By Lemma \ref{lem-it}(1), we have $\ITdist (A) \ge 1.$
			On the other hand, by the below Theorem  \ref{theo-bound}, we get $\ITdist(A) \le \ell\ell(A)-2=1$. Also $\ell\ell(A)=3$. Thus $\ITdist(A)=1$.
	}\end{Bsp}
	
	\subsection{Finiteness of Igusa-Todorov distances}
	In this subsection, we provides an upper bound for the Igusa-Todorov distance of an Artin algebra. We first recall some basic results about the Layer lengths, as detailed in reference \cite{hlh13}.
	
	Let $A$ be an Artin algebra. For an $A$-module $M$, the radical and top of $M$ are denoted by $\rad(M)$ 
	and $\top(M)$, respectively.
	For a subclass $\mathcal{X}$ of $A\modcat$, we denote by $\add(\mathcal{X})$ 
	the subcategory
	of $A\modcat$ consisting of direct summands of
	finite direct sums of modules in $\mathcal{X}$,
	and if $\mathcal{X}=\{X\}$ for some $X\in A\modcat$, we write $\add(X):=\add(\mathcal{X})$.
	The projective dimension $\pd(\mathcal{X})$
	of $\mathcal{X}$ is defined as
	\begin{equation*}
		\pd(\mathcal{X}):=
		\begin{cases}
			\sup\{\pd M\mid M\in \mathcal{X}\}, & \text{if} \; \mathcal{X}\neq \varnothing;\\
			-1,&\text{if} \;\mathcal{X}=\varnothing.
		\end{cases}
	\end{equation*}
	
	Let $\mathcal{V}$ be a subset of all simple modules, and $\mathcal{V}'$
	the set of all the others simple modules in $A\modcat$.
	We write
	$$\mathfrak{F}(\mathcal{V}):=\{M\in A\modcat\mid \text{ there exists a chain }
	0\subseteq M_{0}\subseteq M_{1}\subseteq M_{2}\subseteq \cdots\subseteq M_{m-1}\subseteq M_{m}=M$$
	$$\text{ of submodules of  } M \text{ such that each quotients } M_{i}/M_{i-1}\in \mathcal{V}\}, \text{ and }$$
	$$\mathfrak{T}{(\mathcal{V})}:=\{M \in A\modcat \mid\top(M)\in \add(\mathcal{V}')\}.$$
	By \cite[Lemma 5.7 and Proposition 5.9]{hlh13}, 
	$(\mathfrak{T}{(\mathcal{V})}, \mathfrak{F}(\mathcal{V}))$ is a torsion pair.
	We denote by $t_{\mathcal{V}}$ the torsion radical of the torsion pair $(\mathfrak{T}{(\mathcal{V})}, \mathfrak{F}(\mathcal{V}))$.
	
	\begin{Def}{\rm (\cite{hlh13})
			The $t_{\mathcal{V}}$\textit{-radical layer length} is a function
			$\ell\ell^{t_{\mathcal{V}}}:A\modcat \ra \mathbb{N}\cup \{\infty\}$
			via $$\ell\ell^{t_{\mathcal{V}}}(M):=\inf\{i\ge 0\mid t_{\mathcal{V}}\circ F_{t_{\mathcal{V}}}^{i}(M)=0, M\in A\modcat\},$$
			where $F_{t_{\mathcal{V}}}=\rad\circ t_{\mathcal{V}}. $
		}
	\end{Def}
	Let $M$ be an $A$-module. If $\mathcal{V}=\varnothing$, then $\ell\ell(M)$ is the Loewy length of $M$.
	\begin{Theo}\label{theo-bound}
		Let $A$ be an Artin algebra and $\V$ the set of some simple modules with finite projective dimension.
		Then $\ITdist(A)\le \max\{\ell\ell^{t_{\V}}(A)-2,0\}.$ In particular, $\ITdist(A)\le \max\{\ell\ell(A)-2,0\}.$
	\end{Theo}
	\begin{proof}
		If $\ell\ell^{t_{\V}}(A)\le 2$, then
		$A$ is $(\pd(\V)+2)$-Igusa-Todorov algebra by  \cite[Theorem 1.2]{zheng21}. Then $\ITdist(A)=0$ by Lemma \ref{lem-it}. If $\ell\ell^{t_{\V}}(A)> 2$, then $A$ is $(\ell\ell^{t_{\V}}(A)-2,\pd( \V)+2)$-Igusa-Todorov algebra by \cite[Theorem 4.7]{zheng22}. Then  $\ITdist(A)\le \ell\ell^{t_{\V}}(A)-2$ by Definition \ref{def-itdist}. 
	\end{proof}
	
	\subsection{Relation with Rouquier's dimensions}
	In this subsection, we establish an upper bound for the dimension of  the singularity category using the Igusa-Todorov distance. We begin by recalling the basic definition and fundamental properties of the dimension of a triangulated category, as detailed in reference   \cite{ 0pp09,rou06,rou08}.
	
	Let $\T$ be a triangulated category and fix subcategories $\I,\I_1,\I_2$ of $\T$.
	Denote by $\langle \I \rangle_{1}$ the smallest full subcategory of $\T$ which contains $\I$ and is closed under taking finite direct sums, direct summands, and all shifts.
	Denote by $\I_{1}*\I_{2}$ by the full subcategory of all extensions between them, that is,
	$$\I_{1}*\I_{2}=\{ X\in \T \mid  X_{1} \longrightarrow X
	\longrightarrow X_{2}\longrightarrow X_{1}[1]
	\text{ with } X_{1}\in \I_{1} \text{ and }  X_{2}\in \I_{2}\}.$$
	Write $\I_{1}\diamond\I_{2}:=\langle\I_{1}*\I_{2} \rangle_{1}.$
	Then $(\I_{1}\diamond\I_{2})\diamond\I_{3}=\I_{1}
	\diamond(\I_{2}\diamond\I_{3})$ for any subcategory $\I_{3}$ of $\T$ by the octahedral axiom.
	Write
	\begin{align*}
		\langle \I \rangle_{0}:=0,\;
		\langle \I \rangle_{n+1}:=\langle \I
		\rangle_{n}\diamond\langle \I \rangle_{1}\;{\rm for\; any \;}
		n\geqslant 1.
	\end{align*}
	
	\begin{Def}{\rm (\cite[Definiton 3.2]{rou08})\label{def-dim}
			Let $\T$ be a triangulated category.
			The \textit{dimension} of $\T$, denoted by $\dim( \T)$,  is the minimal integer $d\ge 0$ such that there exists  $M\in \T$ with
			$\T=\langle M \rangle_{d+1}$. If no such $M$ exists for any $d$,
			then we set $\dim(\T)=\infty.$
		}
	\end{Def}
	
	Let $A$ be an Artin algebra over a commutative Artin ring $k$. Denote by $A\modcat$ the category of finitely generated left $A$-modules, and by $\Db{A}$ the bounded derived category. We call a complex in $\Db{A}$ \emph{perfect} if it is isomorphic to a bounded complex of finite generated projective modules.
	It is well known that a complex $\cpx{P}\in \Db{A}$ is perfect if and only if the functor $\Hom_{\Db{A}}(\cpx{P},-)$ preserves small coproducts, that is, $\Hom_{\Db{A}}(\cpx{P},\bigoplus_{i\in I}X_i)\cong \bigoplus_{i\in I}\Hom_{\Db{A}}(\cpx{P},X_i)$ for any set $I$. Let $\per(A)$ be the full subcategory of $\Db{A}$ consisting of all perfect complexes.
	Following \cite{Buch86, Orl04}, the \emph{singularity category} of $A$ is defined to be the Verdier quotient $\Dsg(A) = \Db{A}/\per(A)$. Denote by $q:  \Db{A}\rightarrow \Dsg(A)$ the quotient functor.
	
	\begin{Lem}\label{lem-sing}
		{\rm (\cite[Lemma 2.1]{chen11})}
		Let $X^\bullet$ be a complex in $\Dsg(A)$ and $r_0>0$. Then for any $r$ large enough, there exists
		a module $M$ in $\Omega^{r_0}(A\mbox{-{\rm mod}})$ such that $X^\bullet\simeq q(M)[r]$ in $\Dsg(A)$.
	\end{Lem}
	
	\begin{Lem}\label{lem-sing-ex}
		{\rm (\cite[Lemma 2.2]{chen11})}
		Let $0\rightarrow M\rightarrow P^{1-n}\rightarrow \cdots \rightarrow P^0\rightarrow N\rightarrow 0$ be an exact sequence
		with each $P^i$ projective. Then we have an isomorphism $q(N)\simeq q(M)[n]$ in $\Dsg(A)$. In particular, for an
		$A$-module $M$, we have a natural isomorphism $q(\Omega^n(M))\simeq q(M)[-n]$ in $\Dsg(A)$.
	\end{Lem}
	
	\begin{Theo}\label{sing-it}
		Let $A$ be an Artin algebra. Then $\dim(\Dsg(A))\le \ITdist(A)$.
	\end{Theo}
	\begin{proof}
		Let $X^{\bullet}\in \Dsg(A)$. By Lemma \ref{lem-sing}, there an $A$-module $M$ such that $\cpx{X}\simeq q(M)[r]$ in $\Dsg(A)$ for some $r\in \mathbb{Z}$. 
		Set $m:=\ITdist(A)$.
		By Definition \ref{def-itdist}, $A$ is an $(m,n)$-Igusa-Todorov algebra for some $n$.  
		By Lemma \ref{lem-sing-ex}, $q(\Omega^n(M))\simeq q(M)[-n]$ in $\Dsg(A)$. Then $\cpx{X}\simeq q(\Omega^n(M))[n+r]$ in $\Dsg(A)$.
		By Definition \ref{def-itdist},  there is a module
		$U\in A\modcat$ such that there exists an exact sequence
		\begin{equation}\label{eq-sing}
			0\lra U_m\lra U_{m-1}\lra \cdots \lra U_1\lra U_0\lra \Omega^n(M)\lra 0
		\end{equation}
		where $U_i \in \add(U)$ for each $0 \le i \le m $. 
		By the above exact sequence (\ref{eq-sing}), we can get  the following short exact sequences
		\begin{equation*}
			\xymatrix@C=1em@R=0.1em{
				0 \ar[r] &K_{1} \ar[r] &U_{0}\ar[r] & \Omega^{n}(M) \ar[r] & 0\\
				0\ar[r] & K_{2}  \ar[r] & U_{1} \ar[r] & K_{1}  \ar[r] & 0\\
				&& \vdots&\\
				0 \ar[r] &  K_{m-1} \ar[r] &  U_{m-2}\ar[r] & K_{m-2} \ar[r] &  0\\
				0 \ar[r] &  U_{m}\ar[r] &  U_{m-1}\ar[r] &  K_{m-1} \ar[r] &  0.
			}
		\end{equation*}
		Then we have the following triangles in $\Db{A}$:
		\begin{equation*}
			\xymatrix@C=1em@R=0.1em{
				K_{1} \ar[r] &U_{0}\ar[r] & \Omega^{n}(M) \ar[r] & K_{1}[1]\\
				K_{2}  \ar[r] & U_{1} \ar[r] & K_{1}  \ar[r] & K_{2}[1]\\
				& \vdots&\\
				K_{m-1} \ar[r] &  U_{m-2}\ar[r] & K_{m-2} \ar[r] & K_{m-1} [1]\\
				U_{m}\ar[r] &  U_{m-1}\ar[r] &  K_{m-1} \ar[r] &   U_{m}[1].
			}
		\end{equation*}
		Thus we have the following triangles in $\Dsg(A)$:
		\begin{equation*}
			\xymatrix@C=1em@R=0.1em{
				q(K_{1}) \ar[r] &q(U_{0})\ar[r] & q(\Omega^{n}(M)) \ar[r] & q(K_{1})[1]\\
				q(K_{2})  \ar[r] & q(U_{1}) \ar[r] & q(K_{1}) \ar[r] & q(K_{2})[1]\\
				& \vdots&\\
				q(K_{m-1}) \ar[r] &  q(U_{m-2})\ar[r] & q(K_{m-2}) \ar[r] & q(K_{m-1} )[1]\\
				q(U_{m})\ar[r] &  q(U_{m-1})\ar[r] & q(K_{m-1})\ar[r] &   q(U_{m})[1].
			}
		\end{equation*}
		Moreover, we can get the following triangles in $\Dsg(A)$:
		{\footnotesize
			\begin{equation*}
				\xymatrix@C=1em@R=0.1em{
					q(U_{0})[n+r]\ar[r] & q(\Omega^{n}(M))[n+r] \ar[r] & q(K_{1})[n+r+1] \ar[r] &q(U_{0})[n+r+1]\\
					q(U_{1}) [n+r+1]\ar[r] & q(K_{1})[n+r+1] \ar[r] & q(K_{2})[n+r+2] \ar[r] &q(U_{1}) [n+r+2]\\
					& \vdots&&\\
					q(U_{m-2})[n+r+m-2]\ar[r] & q(K_{m-2}) [n+r+m-2]\ar[r] & q(K_{m-1} )[n+r+m-1]\ar[r] &q(U_{m-2})[n+r+m-1]\\
					q(U_{m-1})[n+r+m-1]\ar[r] & q(K_{m-1})[n+r+m-1]\ar[r] &   q(U_{m})[n+r+m]\ar[r]&q(U_{m-1})[n+r+m].
				}
		\end{equation*}}
		Hence $\cpx{X}\simeq q(\Omega^{n}(M))[n+r] \in \langle q(U) \rangle_{m+1}$ and $\dim(\Dsg(A))\le m= \ITdist(A)$.
	\end{proof}
	
	Combining Theorem \ref{theo-bound} with Theorem \ref{sing-it}, we then obtain the following known result in \cite[Theorem 1.2]{zh20}, which is a generalization of \cite[Proposition 3.7]{rou06}. 
	\begin{Koro} 
		Let $A$ be an Artin algebra. Then $\dim(\Dsg(A))\le \max\{\ell\ell^{t_{\V}}(A)-2,0\}.$
	\end{Koro}
	
	\subsection{Relation with weak resolution dimensions}
	
	This subsection is devoted to establishing a link between the weak resolution dimensions and the Igusa–Todorov distances. We begin by recalling the notion of weak resolution dimension, which was introduced by Oppermann in \cite{0pp09} in order to provide a lower bound for the representation dimension. Note that this definition differs from the one introduced by Iyama in \cite{Iyama2003}.
	
	\begin{Def}\label{wrd-def}{\rm (\cite[Defition 2.2 and Defition 2.4]{0pp09})
			Let $A$ be an Artin algebra and $M$ an $A$-module.
			Then the \emph{$M$-resolution dimension} of an $A$-module $X$ is defined to be
			\begin{align*}
				M\mbox{-}\rd(X)&:=\min \{n\in\mathbb{N} \mid \mbox{ there is an exact sequence } \\
				&0 \ra M_{n} \ra M_{n-1}\ra\cdots \ra M_{0}\ra X\ra 0\mbox{ with all } M_{i}\in \add(M) \text{ such that}\\
				&0 \ra \Hom_A(M, M_{n})  \ra\cdots \ra \Hom_A(M, M_{0})  \ra \Hom_A(M, X) \ra 0\text{ is exact} \};
			\end{align*}
			and the \emph{$M$-resolution dimension} a subcategory $\mathscr{X}$ of $A\modcat$ is defined to be
			\begin{center}
				$M$-$\rd(\mathscr{X}):=\sup\{M$-$\rd(X) \mid X\in \mathscr{X}\}$;
			\end{center}
			the \emph{resolution dimension} of a subcategory $\mathscr{X}$ of $A\modcat$ is defined to be
			\begin{center}
				$\rd(\mathscr{X}):=\min\{M$-$\rd(A)\mid M\in A\modcat\}$.
			\end{center}
			The \emph{weak $M$-resolution dimension} of an $A$-module $X$ is defined to be
			\begin{align*}
				M\mbox{-}\wrd(X):=\min& \{n\in\mathbb{N} \mid \mbox{ there is an exact sequence } \\
				&0 \ra M_{n} \ra M_{n-1}\ra\cdots \ra M_{0}\ra X\ra 0 \\
				&\mbox{with all } M_{i}\in \add(M)\};
			\end{align*}
			and the \emph{weak $M$-resolution dimension} a subcategory $\mathscr{X}$ of $A\modcat$ is defined to be
			\begin{center}
				$M$-$\wrd(\mathscr{X}):=\sup\{M$-$\wrd(X) \mid X\in \mathscr{X}\}$;
			\end{center}
			the \emph{weak resolution dimension} of a subcategory $\mathscr{X}$ of $A\modcat$ is defined to be
			\begin{center}
				$\wrd(\mathscr{X}):=\min\{M$-$\wrd(A)\mid M\in A\modcat\}$.
			\end{center}
	}\end{Def}
	
	\begin{Def}{\rm (\cite[Chapter III, Section 5]{aus71})}
		Let $A$ be an Artin algebra. If $A$ is not semisimple, then the representation dimension of $A$ is defined by $$\rep(A):=\min\{\gd(\End_A(M))\mid M \text{ is generator-cogenerator}\}.$$
		Here an $A$-module $M$ is called a \textit{generator-cogenerator} if every indecomposable projective module and also every indecomposable injective module is isomorphic to a summand of $M$.
	\end{Def}
	When $A$ is semisimple, Auslander assigns the representation dimension $0$ whereas we define it to be $2$ here. 
	
	\begin{Lem}{\rm (\cite[Lemma 2.1]{ehis04})}\label{lem-rep-wrd}
		Let $A$ be an Artin algebra. Then
		$$\rep(A):=\min\{M\text{-}\rd(A\modcat)\mid M \text{ is generator-cogenerator}\}+2.$$ Moreover, $\wrd(A\modcat)+2\le  \rep(A).$
	\end{Lem}
	
	Recall that a complex $\cpx{X}=(X^i,d_X^i)\in\mathscr{C}(A)$ is said to be \emph{exact} if the cohomology vanishes in all degrees, i.e., $H^i(\cpx{X})=0$ for all $i$. It is called \emph{totally exact} if it is exact and the complex $\Hom_A(\cpx{X},A)$ is exact. Let $X$ be an $A$-module. An exact complex $\cpx{P}\in \mathscr{C}(\pmodcat{A})$ is called a \emph{complete projective resolution }of $X$ if ${\rm Ker}(d_P^0)= X$. By a \emph{total projective resolution} of $X$ we mean a totally exact, complete projective resolution of $X$.  Following \cite{EJ95}, the module $_AX$ is called \emph{Gorenstein-projective} if it admits a total projective resolution. In $A\modcat$, Gorenstein-projective modules concide with the modules of $G$-dimension $0$ in the sense of Auslander-Bridge \cite{AB69}.
	We denote by $A\text{-}\mathrm{Gproj}$ the full subcategory of $A\modcat$ consisting of all finitely generated Gorenstein-projective $A$-modules. It is known that $A\text{-}\mathrm{Gproj}$ contains $\pmodcat{A}$ and $\Omega^i(A\text{-}\mathrm{Gproj})=A\text{-}\mathrm{Gproj}$ for $i\ge 0$.
	
	Recall that an Artin algebra $A$ is
	\emph{Gorenstein} provided that the
	regular module $A$ has finite injective dimension on both sides (\cite{Hap2}).  It follows from \cite[Lemma 6.9]{AR} that for a Gorenstein algebra $A$ we have $\id
	(_AA)=\id(A_A)$. If $\id
	(_AA)\le s$, we say
	that $A$ is \emph{$s$-Gorenstein}.
	
	\begin{Lem}\label{lem-gor}{\rm (\cite[Theorem 3.1]{am02})}
		Let $A$ be an $s$-Gorenstein algebra. Then
		$\Omega^s(A\modcat)= A\text{-}\mathrm{Gproj}$.
	\end{Lem}
	
	\begin{Theo}\label{theo-wrd}
		Let $A$ be a Gorenstein algebra. Then 
		$$\ITdist(A)=\wrd(A\text{-}\mathrm{Gproj}).$$ 
		In particular, if $A$ is selfinjective, then $\ITdist(A)=\wrd(A\modcat)$.
	\end{Theo}
	\begin{proof}
		Assume $A$ be an $s$-Gorenstein algebra. 
		
		Firstly, we prove $\ITdist(A)\le \wrd(A\text{-}\mathrm{Gproj})$. Indeed, since $A$ is an $s$-Gorenstein algebra, for each $A$-module $M$, it follows from Lemma \ref{lem-gor} that $\Omega^s(M)\in A\text{-}\mathrm{Gproj}$. 
		Set $w:=\wrd(A\text{-}\mathrm{Gproj})$. By Definition \ref{wrd-def}, there is a module
		$W\in A\modcat$ such that we have an exact sequence
		$$0\lra W_w\lra W_{w-1}\lra \cdots \lra W_1\lra W_0\lra \Omega^s(M)\lra 0,$$
		where $W_i \in \add(W)$ for each $0 \le i \le w $.
		Then, by Definition \ref{def-itdist}, we obtain
		$\ITdist(A)\le w=\wrd(A\text{-}\mathrm{Gproj})$.
		
		Next, we prove $\ITdist(A)\ge \wrd(A\text{-}\mathrm{Gproj})$. Set $m:=\ITdist(A)$.
		By Definition \ref{def-itdist}, $A$ is an $(m,n)$-Igusa-Todorov algebra for some $n$.  Then there is a module
		$U\in A\modcat$ such that for each module $X\in A\modcat$ there exists an exact sequence
		$$0\lra U_m\lra U_{m-1}\lra \cdots \lra U_1\lra U_0\lra \Omega^n(X)\lra 0$$
		where $U_i \in \add(U)$ for each $0 \le i \le m $. Notice $\Omega^n(A\text{-}\mathrm{Gproj})=A\text{-}\mathrm{Gproj}$. By Definition \ref{wrd-def}, $\wrd(A\text{-}\mathrm{Gproj})\le m=\ITdist(A)$.
	\end{proof}
	
	The following example implies the Igusa-Todorov distance may be very large.
	\begin{Bsp}\label{ex-ex}
		{\rm
		Let $n\geqslant 1$ be an integer and $A:=\bigwedge(k^n)$ exterior algebras.
		By \cite[Theorem 4.1]{rou06},
		$$\dim(\stmodcat{A})=\rep(A)-2=n-1.$$ 
		Here $\stmodcat{A}$ is the stable module category of $A$ (see Section \ref{sect-st} for
		a detailed discussion).
		By Theorem \ref{sing-it}, $\dim(\Dsg(A))\le \ITdist(A)$. Note that $A$ is selfinjective.  It follows from $\stmodcat{A}\simeq \Dsg(A)$ as triangulated categories that $\dim(\stmodcat{A})\le \ITdist(A)$. 
		By Theorem \ref{theo-wrd}, $\ITdist(A)=\wrd(A\modcat)$.
		By Lemma \ref{lem-rep-wrd}, $\wrd(A\modcat)\le  \rep(A)-2.$ Then we have 
		$$n-1=\dim(\stmodcat{A})\le \ITdist(A)=\wrd(A\modcat)\le \rep(A)-2=n-1.$$ 
		Thus $$\dim(\stmodcat{A})= \ITdist(A)=\wrd(A\modcat)=\rep(A)-2=n-1.$$ 
	}
	\end{Bsp}
	
	\section{Stable equivalences and  Igusa-Todorov distances}\label{sect-st}
	
	In this section, we prove that stable equivalences of algebras without nodes preserve their Igusa-Todorov distances. We first recall some basic results about the stable equivalence of Artin algebras , as detailed in reference \cite{ars97,c21,Guo05,Xi2022}.
	
	Let $A$ be an Artin algebra. Recall that a simple $A$-module $S$ is called a \emph{node} of $A$ if it is  neither projective nor injective, and the middle term of the almost split sequence starting at $S$ is projective. An $A$-module $X$ is called a \emph{generator} if $A\in\add(X)$.  Denote by $\stmodcat{A}$ the stable module category of $A$ modulo projective modules. The objects are the same as the objects of $A\modcat$, and for two modules $X, Y$ in $\stmodcat{A}$, their homomorphism set is $\underline{\Hom}_A(X,Y):=\Hom_A(X,Y)/\mathscr{P}(X,Y)$, where $\mathscr{P}(X,Y)$ is the subgroup of $\Hom_A(X,Y)$ consisting of the homomorphisms factorizing through a projective $A$-module. This category is usually called the \emph{stable module category} of $A$. Dually, We denoted by  $A\mbox{-}\overline{\mbox{mod}}$ the stable module category of $A$ modulo injective modules.
	Two algebras $A$ and $B$ are said to be \emph{stably equivalent} if the two stable categories $A\stmodcat$ and $B\stmodcat$ are equivalent as additive categories.
	
	Now, let $A$ and $B$ be Artin algebras without nodes. Suppose that $F: \stmodcat{A}\ra \stmodcat{B}$ is an equivalence.
	Then there are one-to-one correspondences
	$$F: A\modcat_{\mathscr{P}}\lra B\modcat_{\mathscr{P}},$$
	where $A\modcat_{\mathscr{P}}$ stands for the full subcategory of $A\modcat$ consisting of modules without nonzero projective summands.
	We also use $F$ to denote the induce map $A\modcat\ra B\modcat$ which takes projectives to zero.
	
	Recall that an exact sequence $0\ra X\lraf{f} Y\ra Z\lraf{g} 0$ in $A\modcat$ is called \emph{minimal} (\cite{MV1990}) if it has no a split exact sequence as a direct summand, that is, there does not exist isomorphisms $u$, $v$, $w$ such that the following diagram
	$$\xymatrix{
		0\ar[r]
		&X\ar[r]^{f}\ar[d]^{u}
		&Y\ar[r]^g\ar[d]^{v}
		&Z\ar[r]\ar[d]^{w}
		&0\\
		0\ar[r]&X_1\oplus
		X_2\ar[r]^{\text{
				$\left(
				\begin{smallmatrix}
					f_1&0\\
					0&f_2
				\end{smallmatrix}\right)
				$}}
		&Y_1\oplus Y_2\ar[r]^{\text{
				$\left(
				\begin{smallmatrix}
					g_1&0\\
					0&g_2
				\end{smallmatrix}\right)
				$}}&Z_1\oplus Z_2\ar[r]&0
	}$$
	is row exact and commute, where $Y_2\neq 0$ and $0\ra X_2\lraf{f_2} Y_2\lraf{g_2}Z_2\ra 0$ is split. The next lemma shows that the stable functor has certain ``exactness" property.
	
	\begin{Lem}\label{lem-seq}{\rm\cite[Theorem 1.7]{MV1990}} Let $0\ra X\oplus P_1\lraf{f} Y\oplus P\lraf{g} Z\ra 0$ be a minimal exact sequence of $A$-modules, where $X,Y,Z\in A\modcat_{\mathscr{P}}$ and $P_1,P\in \pmodcat{A}$. Then
		there is a minimal exact sequence
		$$0\lra F(X)\oplus Q_1\lraf{f'} F(Y)\oplus Q\lraf{g'}F(Z)\lra 0$$
		in $B\modcat$ with $Q_1,Q\in \pmodcat{B}$ and $\underline{g'} =F(\underline{g})$. In particular, $\Omega_{B}(F(Z))\simeq F(\Omega_A(Z))$ in $\stmodcat{B}$ for $Z\in A\modcat_{\mathscr{P}}$.
	\end{Lem}
	
	\begin{Lem}\label{lem-mwrd}
		Let $_AU$ be an $A$-module and $V:=F(U)\oplus B$. Then for each $X\in A\modcat$, we have $$V\text{-}\wrd(F(X))\le U\text{-}\wrd(X).$$
	\end{Lem}
	\begin{proof}
		This can be proved by induction on $U\text{-}\wrd(X)$. 
		In fact, if $U\text{-}\wrd(X)=0$, then $X\simeq U_0$ in $A\modcat$ with $U_0\in \add(U)$. 
		Thus $F(X)\simeq F(U_0)$ in $B\modcat$ and $V\text{-}\wrd(F(X))=0$. Now suppose that for each $_AX$ with $0\le U\mbox{-}\wrd(_AX)\le m-1$, we have $V\mbox{-}\wrd(_BF(X))\le U\mbox{-}\wrd(_AX)$. We shall show the conclusion for $_AX$ with $U\mbox{-}\wrd(_AX)=m$. By Definition \ref{wrd-def}, there exists an exact sequence
		$$0\lra U_m\lraf{f_m} U_{m-1} \lraf{f_{m-1}} \cdots \lraf{f_1} U_0 \lraf{f_0} X\lra 0$$
		in $A\modcat$ with $U_i\in\add(_AU)$ for $0\le i\le m$. Let $K$ is the kernel of $f_0$. Then $U\mbox{-}\wrd(_AK)\le m-1$ and we have a
		short exact sequence \begin{align}\label{kmk}
			0\lra K\lra U_0\lra X \lra 0 
		\end{align}
		in $A\modcat$. 
		We can decompose the short exact sequence (\ref{kmk}) as the direct sums of the following two short exact sequences
		\begin{align}\label{uwu}
			0\lra K_1\lra W_1\lra L_1\lra 0,
		\end{align}
		\begin{align}\label{vzv}
			0\lra K_2\lra W_2\lra L_2\lra 0
		\end{align}
		in $A\modcat$, namely there are isomorphisms $\lambda$, $\mu$, $\nu$ with the following commutative diagram
		\begin{align}\label{2row}
			\xymatrix@C=0.5em@R=0.5em{
				0\ar[rr]
				&&K\ar[rr]\ar[dd]^{\lambda}
				&&U_0\ar[rr]\ar[dd]^{\mu}
				&&X\ar[rr]\ar[dd]^{\nu}
				&&0\\
				&&&&&&&&\\
				0\ar[rr]
				&&K_1\oplus
				K_2\ar[rr]
				&&W_1\oplus W_2\ar[rr]
				&&L_1\oplus L_2\ar[rr]
				&&0,
		}\end{align}
		in $A\modcat$ such that (\ref{uwu}) is minimal and (\ref{vzv}) is split.
		We write $K_1:=K^{^{\mathscr{P}}}_1\oplus P_1$ and $W_1:=W^{^{\mathscr{P}}}_1\oplus P$ with $K^{^{\mathscr{P}}}_1,W^{\mathscr{P}}_1\in A\modcat_{\mathscr{P}}$ and $P_1, P\in \pmodcat{A}$.
		Then we can write the following short exact sequence
		\begin{align}\label{uuu}
			0\lra K^{^{\mathscr{P}}}_1\oplus P_1\lra W^{^{\mathscr{P}}}_1\oplus P\lra L_1\lra 0
		\end{align}
		for the sequence (\ref{uwu}). Since the sequence (\ref{uwu}) is minimal, we know that the sequence (\ref{uuu}) is minimal and $L_1 \in A\modcat_{\mathscr{P}}$. By Lemma \ref{lem-seq}, we have the following minimal exact sequence
		\begin{align}\label{fff} 
			0\lra F(K^{^{\mathscr{P}}}_1)\oplus Q_1\lra F(W^{^{\mathscr{P}}}_1)\oplus Q\lra F(L_1)\lra 0
		\end{align}
		in $B\modcat$ such that $Q_1,Q\in \pmodcat{B}$. 
		
		By induction hypothesis, it follows from $U\mbox{-}\wrd(_AK)\le m-1$  that $V\mbox{-}\wrd(F(K))\le m-1$. Then there exists an exact sequence
		\begin{align}\label{nextk}
			0\lra V_{m-1} \lraf{g_{m-1}} V_{m-2}\lra \cdots \lra V_1\lraf{g_1} V_0 \lraf{g_0} F(K)\lra 0
		\end{align}
		in $B\modcat$ with $V_i\in\add(_BV)$ for $0\le i\le m-1$. 
		Since $\lambda$ is isomorphic, $K\simeq K_1\oplus K_2$ in $A\modcat$. Also $K_1=K_1^{\mathscr{P}}\oplus P_1$ and $P_1\in\pmodcat{A}$. Then $F(K)\simeq F(K_1\oplus K_2)\simeq F(K_1^{\mathscr{P}})\oplus F(K_2)$ in $B\modcat$.
		By the diagram \ref{2row}, the exact sequences (\ref{fff}) and (\ref{nextk}), we get the following long exact sequence
		{\footnotesize
			\begin{align}\label{long-u}
				0\lra V_{m-1} \lraf{g_{m-1}} V_{m-2}\lra \cdots \lra V_1\lraf{g_1} V_0\oplus Q_1 \lra F(W^{^{\mathscr{P}}}_1)\oplus Q\oplus F(K_2)\oplus F(L_2)\lra F(L_1)\oplus F(L_2)\lra 0
		\end{align}}
		Since $\mu$ is isomorphic, $U_0\simeq W_1\oplus W_2$. Since (\ref{vzv}) is split, $W_2\simeq K_2\oplus L_2$. It follows from $U_0\in \add(_AU)$ and $W_1=W^{^{\mathscr{P}}}_1\oplus P$ that $W^{^{\mathscr{P}}}_1,K_2,L_2\in \add(_AU))$. Then $F(W^{^{\mathscr{P}}}_1),F(K_2),F(L_2)\in \add(_BV)$. Since $\nu$ is isomorphic, $X\simeq L_1\oplus L_2$ in $A\modcat$. Then $F(X)\simeq F(L_1)\oplus F(L_2)$ in $B\modcat$.  By the sequence (\ref{long-u}), $V\mbox{-}\wrd(X)\le m$. Hence our claim is proved.
		
	\end{proof}
	\begin{Theo}\label{st-thm}
		Let $A$ and $B$ be stably equivalent Artin algebras without nodes. Then $$\ITdist(A)=\ITdist(B).$$
	\end{Theo}
	\begin{proof}
		We first prove $\ITdist(B)\le \ITdist(A)$.  For $Y\in B\modcat$, we write $Y=Y^{\mathscr{P}}\oplus Q'$ with $Y^{\mathscr{P}}\in B\modcat_{\mathscr{P}}$ and $Q'\in \pmodcat{B}$. Since $F: A\modcat_{\mathscr{P}}\ra B\modcat_{\mathscr{P}}$ is an one-to-one correspondence, there exists $X\in A\modcat_{\mathscr{P}}$ such that $F(X)\simeq Y^{\mathscr{P}}$ as $B$-modules.
		
		Set $m:=\ITdist(A)$. By Definition \ref{def-itdist}, $A$ is a $(m,n)$-Igusa-Todorov algebra for some $n$. By Definition \ref{def-italg}, there is an $(m,n)$-Igusa-Todorov module $_AU$ such that for each $A$-module $X$, we have the following exact sequence
		\begin{align*}
			0\lra U_m\lra \cdots \lra U_1\lra U_0\lra \Omega_A^n(X)\lra 0
		\end{align*}
		in $A\modcat$ with $U_i\in \add(_AU)$ for $0\le i\le m$. 
		Then $U\text{-}\wrd(\Omega_A^n(X))\le m$.
		Define $V:=F(U)\oplus B$. 
		By Lemma \ref{lem-mwrd}, $V\text{-}\wrd(F(\Omega_A^n(X))\le m$.
		Then there exists an exact sequence
		\begin{align}\label{nextk-2}
			0\lra V_{m} \lra V_{m-1}\lra \cdots \lra V_1\lra V_0 \lra  F(\Omega_A^n(X))\lra 0
		\end{align}
		in $B\modcat$ with $V_i\in\add(_BV)$ for $0\le i\le m$. 
		By Lemma \ref{lem-seq}, $F(\Omega_A^n(X))\simeq \Omega_B^n(F(X))$ in $\stmodcat{B}$. Then there are projective $B$-modules $Q$ and $Q_1$ such that $F(\Omega_A^n(X))\oplus Q\simeq \Omega_B^n(F(X))\oplus Q_1$ in $B\modcat$. 
		By the sequence \ref{nextk-2}, we have an exact sequence
		\begin{align}\label{nextk-3}
			0\lra V_{m} \lra V_{m-1}\lra \cdots \lra V_1\lra V'_0\oplus Q' \lra \Omega_B^n(F(X))\lra 0
		\end{align}
		where $V'_0$ is a direct sumand of $V_0$ and $Q'\in \pmodcat{B}$.
		Then $V\text{-}\wrd(\Omega_B^n(F(X))\le m$. 
		Note that $Y=Y^{\mathscr{P}}\oplus Q'$ and $F(X)\simeq Y^{\mathscr{P}}$. 
		Thus $V\text{-}\wrd(Y)\le m$. 
		Hence $B$ is a $(m,n)$-Igusa-Todorov algebra and $\ITdist(B)\le m=\ITdist{A}$.
		
		Similarly, we have $\ITdist{A}\le \ITdist(B)$. 
		Thus $\ITdist{A}= \ITdist(B)$.
	\end{proof}
	
	\section{Singular equivalences and  Igusa-Todorov distances}\label{sect-t-level}
	
	In this section, we shall prove that the Igusa-Todorov distance is an invariant under singular equivalence of Morita type with level. We first recall some basic results about the singular equivalences of Morita type with level, as detailed in reference \cite{Wang15}.
	
	Let $A$ be a finite dimensional algebra and $A^e = A\otimes _k A^{\opp}$ the enveloping algebra of $A$. We identify
	$A$-$A$-bimodules with left $A^e$-modules. Denote by $\Omega _{A^e}(-)$ the syzygy functor on the
	stable category $A^e\stmodcat$ of $A$-$A$-bimodules. The following terminology is
	due to Wang \cite{Wang15}.
	
	\begin{Def}\label{def-level}
		{\rm Let $A$ and $B$ be finite dimensional algebras. Let $_AM_B$ and $_BN_A$ be an $A$-$B$-bimodule and a $B$-$A$-bimodule,
			respectively, and let $l \geq 0$. We say $(M,N)$ defines a {\it singular equivalence
				of Morita type with level} $l$, provided that the following conditions are satisfied:
			
			(1) The four one-sided modules $_AM$, $M_B$, $_BN$ and $N_A$ are all finitely generated projective.
			
			(2) There are isomorphisms $M\otimes _B N \simeq \Omega _{A^e}^l(A)$ and
			$N\otimes _A M\simeq \Omega _{B^e}^l(B)$ in
			$A^e\stmodcat$ and $B^e\stmodcat$, respectively.}
	\end{Def}
	
	\begin{Lem}\label{lem-exact}
		Let $A$ and $B$ be finite dimensional algebras, and $F:A\modcat\ra B\modcat$ be an exact functor. If $F$ preserves projective modules, then for each $A$-module $X$ and each non-positive integer $n$, we have $F(\Omega_A^n(X))\simeq \Omega^n_B(F(X))\oplus Q$ in $B\modcat$ for some projective $B$-module $Q$.
	\end{Lem}
	
	\begin{Theo}\label{thm-level}
		Let $A$ and $B$ be finite dimensional algebras
		If $A$ and $B$ are singularly equivalent of Morita type with level, then $$\ITdist(A)=\ITdist(B).$$
	\end{Theo}
	\begin{proof} Set $m:=\ITdist(A)$. By Definition \ref{def-itdist}, $A$ is a $(m,n)$-Igusa-Todorov algebra for some $n$.
		
		Given $Y\in B\modcat$, we have $M\otimes_AY\in A\modcat$. By Definition \ref{def-italg}, there is an $(m,n)$-Igusa-Todorov module $_AU$ such that we have the following exact sequence
		\begin{align}\label{my-long}
			0\lra U_m\lra \cdots \lra U_1\lra U_0\lra \Omega_A^n(M\otimes_BY)\lra 0
		\end{align}
		in $A\modcat$ with $U_i\in \add(_AU)$ for $0\le i\le m$. Note that $N_A$ is projective. Then we have the following exact sequence
		\begin{align}\label{nmy-long}
			0\lra N\otimes_AU_m\lra \cdots \lra N\otimes_AU_1\lra N\otimes_AU_0\lra N\otimes_A\Omega_A^n(M\otimes_BY)\lra 0
		\end{align}
		in $B\modcat$. By Lemma \ref{lem-exact},
		\begin{align}\label{iso0-nmy}
			N\otimes_A\Omega_A^n(M\otimes_BY)\simeq \Omega_B^n(N\otimes_AM\otimes_BY)\oplus Q_0
		\end{align} for some projective $B$-module $Q_0$.
		By Definition \ref{def-level}, we have $N\otimes_AM\simeq \Omega^l_{B^e}(B)$ in $\stmodcat{B^e}$. By \cite[Theorem 2.2]{h60}, there are projective $B^e$-modules $P$ and $Q$ such that
		$N\otimes_AM\oplus P\simeq \Omega^l_{B^e}(B)\oplus Q$ as $B^e$-modules.
		Then $$N\otimes_AM\otimes_BY\oplus P\otimes_BY\simeq \Omega^l_{B^e}(B)\otimes_BY\oplus Q\otimes_BY$$ as $B$-modules.
		Due to $P,Q\in \pmodcat{B^e}$, $P\otimes_BY$ and $Q\otimes_BY$ are projective $B$-modules. Then
		\begin{align}\label{iso-nmy}
			N\otimes_AM \otimes_BY \simeq
			\Omega^l_{B^e}(B)\otimes_BY
		\end{align} in $\stmodcat{B}$.
		Note that we have the following exact sequence
		\begin{align}\label{be-seq}
			0\lra \Omega^{l}_{B^e}(B)\lra R_{l-1}\lra\cdots \lra R_1\lra R_0\lra B\lra 0
		\end{align}
		in $B^e\modcat$ with projective $B^e$-modules $R_j$ for $0\le j\le l-1$.
		It follows from $B_B\in \pmodcat{B^{\opp}}$ that the exact sequence (\ref{be-seq}) splits in $B^{\opp}\modcat$. Then we have the following exact sequence
		\begin{align}\label{y-seq}
			0\lra \Omega^{l}_{B^e}(B)\otimes_BY\lra R_{l-1}\otimes_BY\lra\cdots \lra R_1\otimes_BY\lra R_0\otimes_BY\lra Y\lra 0
		\end{align}
		in $B\modcat$. Note that $R_j\otimes_BY\in \pmodcat{B}$ for each $0\le j\le l-1$. Then $\Omega^{l}_{B^e}(B)\otimes_BY\simeq \Omega^{l}_B(Y)$ in $B\stmodcat$. By (\ref{iso-nmy}), we obtain that
		\begin{align}\label{iso2-nmy}
			N\otimes_AM \otimes_BY \simeq
			\Omega^{l}_B(Y)
		\end{align} in $\stmodcat{B}$. Then we have
		\begin{align}\label{iso3-nmy}
			\Omega_B^n(N\otimes_AM \otimes_BY) \simeq
			\Omega^{n+l}_B(Y)
		\end{align} in $\stmodcat{B}$. By \cite[Theorem 2.2]{h60}, there are projective $B$-modules $Q_1$ and $Q_2$ such that
		\begin{align}\label{iso4-nmy}
			\Omega_B^n(N\otimes_AM \otimes_BY) \oplus Q_1 \simeq
			\Omega^{n+l}_B(Y) \oplus Q_2
		\end{align} in $B\modcat$. By (\ref{iso0-nmy}) and (\ref{iso3-nmy}),
		\begin{align}\label{iso5-nmy}
			N\otimes_A\Omega_A^n(M\otimes_BY)\simeq \Omega^{n+l}_B(Y) \oplus Q_3
		\end{align} in $B\modcat$ with a projective $B$-module $Q_3=Q_0\oplus Q_2$.
		Hence (\ref{nmy-long}) can be written as
		\begin{align}\label{nmy2-long}
			0\lra N\otimes_AU_m\lra \cdots \lra N\otimes_AU_1\lra V_0\lra \Omega_B^n(N\otimes_AM\otimes_BY)\lra 0
		\end{align}
		where $V_0\oplus Q_3\simeq N\otimes_AU_0$ in $B\modcat$. Set $V:=N\otimes_AU\oplus B$. Then $V_0\in \add(_BV)$ and $N\otimes_AU_i\in\add(_BV)$ for $1\le i\le m$. By Definition \ref{def-italg} and Definition \ref{def-itdist}, $B$ is a $(m,n)$-Igusa-Todorov algebra and $\ITdist(B)\le m=\ITdist(A)$. Similarly, we have $\ITdist(A)\le \ITdist(B)$. Then we get the disired result.
	\end{proof}
	
	By Lemma \ref{lem-it}, Theorem \ref{thm-level} generalizes \cite[Proposition 5.1]{qin22}.
	\begin{Koro}
		Let $A$ and $B$ be Let $A$ and $B$ be finite dimensional algebras which are singularly equivalent of Morita type with level.
		Then $A$ is syzygy-finite $($respectively, Igusa-Todorov$)$ if and only if $B$ is syzygy-finite $($respectively, Igusa-Todorov$).$
	\end{Koro}
	
	Recall that an extension $B\subseteq A$ of finite dimensional algebras is bounded if the $B$-$B$-bimodule $A/B$ is $B$-tensor nilpotent, its projective dimension is finite and ${\rm Tor}_i^B(A/B, (A/B)^{\otimes_B j})=0$ for all $i, j\geq 1$. By \cite[Theorem 0.3]{qxzz24} that for a bounded extension $B\subseteq A$, the algebras $A$ and $B$ are singularly equivalent of Morita type with level. 
	
	\begin{Koro}\label{cor-ext}
		Let $B\subseteq A$ be a bounded extension. Then $\ITdist(A)=\ITdist(eAe)$.
	\end{Koro}
	
	Now, let $A$ be a finite dimensional algebra over a field $k$, and $I\subseteq A$ be a two-sided ideal. Following \cite{px06},
	$I$ is a \textit{homological ideal} if the canonical map $A \rightarrow  A/I$ is a homological epimorphism, that is, the naturally induced functor $\Db{A/I\modcat} \ra \Db{A\modcat}$ is fully faithful.
	
	\begin{Koro}\label{cor-homol}
		Let $A$ be a finite dimensional algebra over a field $k$ and let $I\subseteq A$ be a homological ideal which has finite projective dimension as an $A$-$A$-bimodule. Then $\ITdist(A)=\ITdist(A/I)$.
	\end{Koro}
	\begin{proof}
		By \cite[Theorem 3.6]{qin22}, $A$ and $A/I$ are singularly equivalent of Morita type with level. Then the corollary follows from Theorem \ref{thm-level}.
	\end{proof}
	
	\begin{Koro}
		Let $A$ be a finite-dimensional $k$-algebra such that
		$A/\rad(A)$ is separable over $k$. Let $e \in A$ be an idempotent and assume
		$Ae\otimes _{eAe}^LeA$ is bounded in
		cohomology. If $\pd_A(\frac{A/AeA}{\rad (A/AeA)})<\infty$ or $\id _A(\frac{A/AeA}{\rad (A/AeA)})<\infty$, then $\ITdist(A)=\ITdist(eAe)$.
	\end{Koro}
	\begin{proof}
		By the proof of \cite[Corollary 5.4]{qin22}, we know that $A$ and $eAe$ are singularly equivalent of Morita type with level. Then the corollary follows from Theorem \ref{thm-level}.
	\end{proof}

	\begin{Bsp}{\rm (\cite[Example 3.3]{chen14})
			Let $A$ be the $k$-algebra given by the following quiver with relations.
			$$\begin{array}{cc}
				\xymatrix{1 \ar@<+0.5ex>[r]^\alpha & 3
					\ar@<+0.5ex>[l]^\beta \ar@(ul,ur)^x \ar@<+0.5ex>[r]^\delta & 2
					\ar@<+0.5ex>[l]^\gamma}\\
				x^2=\delta x=\beta x=x\gamma=x \alpha=\beta \gamma=\delta \alpha=\beta \alpha=\delta \gamma=0, \alpha \beta=\gamma \delta.
			\end{array}$$
			We claim $\ITdist(A)=0$. Indeed, let $e$ be the primitive idempotent corresponding to the vertex $1$.
			It follows from \cite[Example 3.3]{chen14} that $AeA$ is a homological ideal which is projective as an $A$-$A$-bimodule. By Corollary \ref{cor-homol}, $\ITdist(A)=\ITdist(A/AeA)$. Note that $A/AeA$ is monomial. By Zimmermann-Huisgen's result in \cite{zh92}, we obtain that $A/AeA$ is $2$-syzygy finite. Then $\ITdist(A)=\ITdist(A/AeA)=0$ by Lemma \ref{lem-it}.
	}\end{Bsp}
	
	\section{Recollements and Igusa-Todorov distances}\label{sect-t-rec}
	This section is devoted to proving Theorem \ref{main-thm-rec}. We first extend the Igusa-Todorov distance for algebras to the context of subclasses of $\Db{A\modcat}$.
	
	\begin{Def}\label{def-r-hered}{\rm
			Let $A$ be an Artin algebra and $\mathcal{C}\subseteq \Db{A\modcat}$. Let $m\in \mathbb{N}$ and $n\in \mathbb{Z}$.
			
			(1) For $m\ge 1$, we say that $\mathcal{C}$ is \textit{relatively} $m$\textit{-hereditary} provided that there is a complex
			$\cpx{U}\in \Db{A\modcat}$ such that for any $\cpx{X}\in \mathcal{C}$ there exist triangles
			$$\xymatrix@R=.2cm@C=.4cm{
				\cpx{X}_1\ar[r]& \cpx{U}_0\ar[r]& \cpx{X}\ar[r]& \cpx{X}_1[1],\\
				\cpx{X}_2\ar[r]& \cpx{U}_1\ar[r]& \cpx{X}_1\ar[r]& \cpx{X}_2[1],\\
				\cpx{X}_3\ar[r]& \cpx{U}_2\ar[r]& \cpx{X}_2\ar[r]& \cpx{X}_3[1],\\
				& \vdots& \vdots& \\
				\cpx{X}_m\ar[r]& \cpx{U}_{m-1}\ar[r]& \cpx{X}_{m-1}\ar[r]& \cpx{X}_m[1],
			}$$
			where $\cpx{U}_i\in \add(\cpx{U})$ for $0\le i\le m-1$ and $\cpx{X}_m\in\add(\cpx{U})$. For $m=0$, we say that $\mathcal{C}$ is \textit{relatively} $0$\textit{-hereditary} if there is $\cpx{U}\in \Db{A\modcat}$ such that $\mathcal{C}=\add(\cpx{U})$.
			
			(2) We say that $\mathcal{C}$ of $\Db{A\modcat}$ is $(m,n)$-\textit{Igusa-Todorov class} if $\Omega_{\mathscr{D}}^n(\mathcal{C})$ is relatively $m$-hereditary.
		}
	\end{Def}
	
	\begin{Def}\label{def-itdist-c}
		{\rm
			Let $A$ be an Artin algebra and $\mathcal{C}\subseteq \Db{A\modcat}$.
			We set the \textit{Igusa-Todorov distance} of $\mathcal{C}$ as follows
			$$\ITdist(\mathcal{C}):=\inf\{m \mid \mathcal{C} \mbox{ is an } (m,n) \mbox{-Igusa-Todorov class for some } n\}.$$
		}
	\end{Def}
	\begin{Lem}\label{lem-itdist-c}
		Let $A$ be an Artin algebra, and $\mathcal{C}\subseteq \Db{A\modcat}$. Then
		
		{\rm (1)} $\ITdist(\mathcal{C})=\ITdist(\mathcal{C}[s])$ for any/some $s\in \mathbb{Z}$.
		
		{\rm (2)} $\ITdist(\mathcal{C})=\ITdist(\Omega^s_{\mathscr{D}}(\mathcal{C}))$ for any/some $s\in \mathbb{Z}$.
		
		{\rm (3)}
		If $\mathcal{E}\subseteq \mathcal{C}$, then $\ITdist(\mathcal{E})\le \ITdist(\mathcal{C})$.
	\end{Lem}
	\begin{proof} Let $m\in \mathbb{N}$ and $n\in \mathbb{Z}$.
		
		(1) It is easy to see that $\mathcal{C}$ is an $(m,n)$-Igusa-Todorov (respectively, relatively $m$-hereditary) class if and only if $\mathcal{C}[s]$ is an $(m,n)$-Igusa-Todorov (respectively, relatively $m$-hereditary) class for any/some $s$. Thus $\ITdist(\mathcal{C})=\ITdist(\mathcal{C}[s])$ for any/some $s\in \mathbb{Z}$.
		
		(2) We first prove that if $\Omega^n_{\mathscr{D}}(\mathcal{C})$ is relatively $m$-hereditary, then $\Omega^s_{\mathscr{D}}(\mathcal{C})$ is relatively $m$-hereditary for any $s\ge n$. Indeed, by Definition \ref{def-r-hered}, there is a complex
		$\cpx{U}\in \Db{A\modcat}$ such that for any $\cpx{X}\in \Omega^n_{\mathscr{D}}(\mathcal{C})$ there exist triangles
		$$\xymatrix@R=.2cm@C=.4cm{
			\cpx{X}_1\ar[r]& \cpx{U}_0\ar[r]& \cpx{X}\ar[r]& \cpx{X}_1[1],\\
			\cpx{X}_2\ar[r]& \cpx{U}_1\ar[r]& \cpx{X}_1\ar[r]& \cpx{X}_2[1],\\
			\cpx{X}_3\ar[r]& \cpx{U}_2\ar[r]& \cpx{X}_2\ar[r]& \cpx{X}_3[1],\\
			& \vdots& \vdots& \\
			\cpx{X}_m\ar[r]& \cpx{U}_{m-1}\ar[r]& \cpx{X}_{m-1}\ar[r]& \cpx{X}_m[1],
		}$$
		where $\cpx{U}_i\in \add(\cpx{U})$ for $0\le i\le m-1$ and $\cpx{X}_m\in\add(\cpx{U})$. By Lemma \ref{lem-omega-shift}, we have triangles
		$$\xymatrix@R=.2cm@C=.4cm{
			\Omega^{s-n}_{\mathscr{D}}(\cpx{X}_1)\ar[r]& \Omega^{s-n}_{\mathscr{D}}(\cpx{U}_0)\oplus P_0\ar[r]& \Omega^{s-n}_{\mathscr{D}}(\cpx{X})\ar[r]& \Omega^{s-n}_{\mathscr{D}}(\cpx{X}_1)[1],\\
			\Omega^{s-n}_{\mathscr{D}}(\cpx{X}_2)\ar[r]& \Omega^{s-n}_{\mathscr{D}}(\cpx{U}_1)\oplus P_1\ar[r]& \Omega^{s-n}_{\mathscr{D}}(\cpx{X}_1)\ar[r]& \Omega^{s-n}_{\mathscr{D}}(\cpx{X}_2)[1],\\
			\Omega^{s-n}_{\mathscr{D}}(\cpx{X}_3)\ar[r]& \Omega^{s-n}_{\mathscr{D}}(\cpx{U}_2)\oplus P_2\ar[r]& \Omega^{s-n}_{\mathscr{D}}(\cpx{X}_2)\ar[r]& \Omega^{s-n}_{\mathscr{D}}(\cpx{X}_3)[1],\\
			& \vdots& \vdots& \\
			\Omega^{s-n}_{\mathscr{D}}(\cpx{X}_m)\ar[r]& \Omega^{s-n}_{\mathscr{D}}(\cpx{U}_{m-1})\oplus P_{m-1}\ar[r]& \Omega^{s-n}_{\mathscr{D}}(\cpx{X}_{m-1})\ar[r]& \Omega^{s-n}_{\mathscr{D}}(\cpx{X}_m)[1],
		}$$
		where $P_i\in \pmodcat{A}$.
		Note that $\Omega^{s-n}_{\mathscr{D}}(\cpx{U}_i)\in \add(\Omega^{s-n}_{\mathscr{D}}(\cpx{U}))$ for $0\le i\le m-1$ and $\Omega^{s-n}_{\mathscr{D}}(\cpx{X}_m)\in\add(\Omega^{s-n}_{\mathscr{D}}(\cpx{U}))$. By Lemma \ref{lem-prop-omega}(4), we have $\Omega^s_{\mathscr{D}}(\mathcal{C})=\Omega^{s-n}_{\mathscr{D}}(\Omega^n_{\mathscr{D}}(\mathcal{C}))$. Thus $\Omega^s_{\mathscr{D}}(\mathcal{C})$ is relatively $m$-hereditary for any $s\ge n$ by Definition \ref{def-r-hered}.
		
		Now assume that $\mathcal{C}$ is an $(m,n)$-Igusa-Todorov class. Then $\Omega^n_{\mathscr{D}}(\mathcal{C})$ is relatively $m$-hereditary. Let $s$ be an integer and take an integer $t>\max\{0,n-s\}$, then we have $\Omega^{t}_{\mathscr{D}}(\Omega^s_{\mathscr{D}}(\mathcal{C}))=\Omega^{s+t}_{\mathscr{D}}(\mathcal{C})$ by Lemma \ref{lem-prop-omega}(4). It follows from $t+s\ge n$ that $\Omega^{t+s}_{\mathscr{D}}(\mathcal{C})$ is relatively $m$-hereditary by the above argument. Thus $\Omega^{s}_{\mathscr{D}}(\mathcal{C})$ is an $(m,t)$-Igusa-Todorov class by Definition \ref{def-itdist-c}. Conversely, assume that $\Omega^{s}_{\mathscr{D}}(\mathcal{C})$ is an $(m,l)$-Igusa-Todorov class for some integer $l$, then $\Omega^{s+l}_{\mathscr{D}}(\mathcal{C})$ is relatively $m$-hereditary. Thus $\mathcal{C}$ is an $(m,s+l)$-Igusa-Todorov class. Hence $\ITdist(\mathcal{C})=\ITdist(\Omega^s_{\mathscr{D}}(\mathcal{C}))$ for any/some $s\in \mathbb{Z}$.
		
		(3) Clearly, if $\mathcal{C}$ is an $(m,n)$-Igusa-Todorov (respectively, relatively $m$-hereditary) class, then $\mathcal{E}$ is also an $(m,n)$-Igusa-Todorov (respectively, relatively $m$-hereditary) class. Thus $\ITdist(\mathcal{E})\le \ITdist(\mathcal{C})$.
	\end{proof}
	
	\begin{Lem}\label{lem-it-equi-1}
		Let $A$ be an Artin algebra. Then $\ITdist(A)=\ITdist(A\modcat)$.
	\end{Lem}
	\begin{proof} Set $\ITdist(A)=m$. By Definition \ref{def-itdist}, there are a nonpostive integer $n$ such that $A$ is $(m,n)$-Igusa-Todorov algebra. Then there is an $A$-module
		$U$ such that for any module $X\in A\modcat$ there exists an exact sequence
		$$0\lra U_m\lraf{f_m} U_{m-1}\lraf{f_{m-1}} \cdots \lra U_1\lraf{f_1} U_0\lraf{f_0} \Omega^n(X)\lra 0$$
		where $U_i \in \add(U)$ for each $0 \le i \le m $. Let $X_{i+1}$ be the kernel of $f_i$ for $0\le i\le m-1$. Then $X_m=U_m$, and we have short exact sequences in $A\modcat$:
		$$\xymatrix@R=.2cm@C=.4cm{
			0\ar[r] &X_1\ar[r] & U_{0}\ar[r] &\Omega^n(X)\ar[r]&0,\\
			0\ar[r] &X_2\ar[r] & U_{1}\ar[r] &X_1\ar[r]&0,\\
			&&\vdots&&\\
			0\ar[r] &X_m\ar[r] & U_{m-1}\ar[r] &X_{m-1}\ar[r]&0.\\
		}$$
		Thus we have triangles in $\Db{A\modcat}$:
		$$\xymatrix@R=.2cm@C=.4cm{
			X_1\ar[r] & U_{0}\ar[r] &\Omega^n(X)\ar[r]&X_1[1],\\
			X_2\ar[r] & U_{1}\ar[r] &X_1\ar[r]&X_2[1],\\
			&\vdots&\vdots&\\
			X_m\ar[r] & U_{m-1}\ar[r] &X_{m-1}\ar[r]&X_m[1].\\
		}$$
		By Definition \ref{def-itdist-c}, we have $\ITdist(A\modcat)\le m= \ITdist(A)$.
		
		Conversely, set $\ITdist(A\modcat)=m'$. By Definition \ref{def-itdist-c}, there are an integer $n'$ such that $A\modcat$ is $(m',n')$-Igusa-Todorov class. By Definition \ref{def-r-hered}, there is a complex
		$\cpx{U}\in \Db{A\modcat}$ such that for any $X\in A\modcat$ there exist triangles
		\begin{equation}\label{triangle-xu}
			\begin{aligned}
				\xymatrix@R=.2cm@C=.4cm{
					\cpx{X}_1\ar[r]& \cpx{U}_0\ar[r]& \Omega^{n'}_{\mathscr{D}}(X)\ar[r]& \cpx{X}_1[1],\\
					\cpx{X}_2\ar[r]& \cpx{U}_1\ar[r]& \cpx{X}_1\ar[r]& \cpx{X}_2[1],\\
					\cpx{X}_3\ar[r]& \cpx{U}_2\ar[r]& \cpx{X}_2\ar[r]& \cpx{X}_3[1],\\
					& \vdots& \vdots& \\
					\cpx{X}_{m'}\ar[r]& \cpx{U}_{m'-1}\ar[r]& \cpx{X}_{m'-1}\ar[r]& \cpx{X}_{m'}[1],
				}
			\end{aligned}
		\end{equation}
		where $\cpx{U}_i\in \add(\cpx{U})$ for $0\le i\le m'-1$ and $\cpx{X}_{m'}\in\add(\cpx{U})$. Assume that $\cpx{X}_1,\cdots,\cpx{X}_m$ and $\cpx{U}\in \mathscr{D}^{[s,t]}(A\modcat)$ for some $s\le t$, then $\Omega_{\mathscr{D}}^r(\cpx{U})\in A\modcat$ for $r\ge \max\{-s,0\}$ by Lemma \ref{lem-prop-omega}(2). By Lemma \ref{lem-omega-shift} and the triangles (\ref{triangle-xu}), we have triangles
		\begin{equation}\label{triangle-xu-omega}
			\begin{aligned}
				\xymatrix@R=.2cm@C=.4cm{
					\Omega_{\mathscr{D}}^r(\cpx{X}_1)
					\ar[r]& \Omega_{\mathscr{D}}^r(\cpx{U}_0)\oplus P_0
					\ar[r]& \Omega_{\mathscr{D}}^r(\Omega^{n'}_{\mathscr{D}}(X))
					\ar[r]& \Omega_{\mathscr{D}}^r(\cpx{X}_1)[1],\\
					\Omega_{\mathscr{D}}^r(\cpx{X}_2)
					\ar[r]&
					\Omega_{\mathscr{D}}^r(\cpx{U}_1)\oplus P_1
					\ar[r]&
					\Omega_{\mathscr{D}}^r(\cpx{X}_1)
					\ar[r]&
					\Omega_{\mathscr{D}}^r(\cpx{X}_2)[1],\\
					\Omega_{\mathscr{D}}^r(\cpx{X}_3)
					\ar[r]&
					\Omega_{\mathscr{D}}^r(\cpx{U}_2)\oplus P_2
					\ar[r]&
					\Omega_{\mathscr{D}}^r(\cpx{X}_2)
					\ar[r]&
					\Omega_{\mathscr{D}}^r(\cpx{X}_3)[1],\\
					& \vdots& \vdots& \\
					\Omega_{\mathscr{D}}^r(\cpx{X}_{m'})
					\ar[r]&
					\Omega_{\mathscr{D}}^r(\cpx{U}_{m'-1})\oplus P_{m'-1}
					\ar[r]& \Omega_{\mathscr{D}}^r(\cpx{X}_{m'-1})
					\ar[r]& \Omega_{\mathscr{D}}^r(\cpx{X}_{m'})[1],
				}
			\end{aligned}
		\end{equation}
		where $P_i\in \pmodcat{A}$ for $0\le i\le m'-1$.
		By Lemma \ref{lem-prop-omega}(4), $\Omega_{\mathscr{D}}^r(\Omega^{n'}_{\mathscr{D}}(X))\simeq \Omega^{r+n'}_{\mathscr{D}}(X)$. Take the integer $r$ such that $r\ge \max\{0,-s,-n'\}$, then all term in the above triangles are $A$-modules. Applying the cohomological functors $H^i:=\Hom_{\Db{A\modcat}}(A,-[i])$, $i=-1,0,1$
		to the triangles (\ref{triangle-xu-omega}), we have short exact sequences
		$$\xymatrix@R=.2cm@C=.4cm{
			0\ar[r]&\Omega_{\mathscr{D}}^r(\cpx{X}_1)
			\ar[r]& \Omega_{\mathscr{D}}^r(\cpx{U}_0)\oplus P_0
			\ar[r]& \Omega_{\mathscr{D}}^r(\Omega^{n'}_{\mathscr{D}}(X))
			\ar[r]& 0,\\
			0\ar[r]&\Omega_{\mathscr{D}}^r(\cpx{X}_2)
			\ar[r]&
			\Omega_{\mathscr{D}}^r(\cpx{U}_1)\oplus P_1
			\ar[r]&
			\Omega_{\mathscr{D}}^r(\cpx{X}_1)
			\ar[r]&
			0,\\
			0\ar[r]&\Omega_{\mathscr{D}}^r(\cpx{X}_3)
			\ar[r]&
			\Omega_{\mathscr{D}}^r(\cpx{U}_2)\oplus P_2
			\ar[r]&
			\Omega_{\mathscr{D}}^r(\cpx{X}_2)
			\ar[r]&
			0,\\
			&\vdots& \vdots& \vdots& \\
			0\ar[r]&\Omega_{\mathscr{D}}^r(\cpx{X}_{m'})
			\ar[r]&
			\Omega_{\mathscr{D}}^r(\cpx{U}_{m'-1})\oplus P_{m'-1}
			\ar[r]& \Omega_{\mathscr{D}}^r(\cpx{X}_{m'-1})
			\ar[r]& 0.
		}$$
		Moreover, we have the following eact sequence in $A\modcat$.
		$$0\ra \Omega_{\mathscr{D}}^r(\cpx{X}_{m'})\ra \Omega_{\mathscr{D}}^r(\cpx{U}_{m'-1})\oplus P_{m'-1}
		\ra \cdots \ra
		\Omega_{\mathscr{D}}^r(\cpx{U}_1)\oplus P_1
		\ra \Omega_{\mathscr{D}}^r(\cpx{U}_0)\oplus P_0
		\ra \Omega_{\mathscr{D}}^r(\Omega^{n'}_{\mathscr{D}}(X))\ra 0$$
		Note that $\Omega_{\mathscr{D}}^r(\cpx{U}_{0})\oplus P_0,\Omega_{\mathscr{D}}^r(\cpx{U}_{1})\oplus P_1,\cdots,\Omega_{\mathscr{D}}^r(\cpx{U}_{m'-1})\oplus P_{m'-1},\Omega_{\mathscr{D}}^r(\cpx{X}_{m'})\in \add(\Omega_{\mathscr{D}}^r(\cpx{U})\oplus A)$.
		By Lemma \ref{lem-prop-omega}(4), $\Omega_{\mathscr{D}}^r(\Omega^{n'}_{\mathscr{D}}(X))\simeq \Omega^{r+n'}_{\mathscr{D}}(X)$. Here $r+n'\ge 0$.
		Thus we obtain that $\ITdist(A)\le m'=\ITdist(A\modcat)$. This finishes the proof.
	\end{proof}
	
	\begin{Lem}\label{lem-it-equi}
		Let $A$ be an Artin algebra. Then for any/some integers $s\le t$, we have $$\ITdist(A)=\ITdist(\mathcal{D}^{[s,t]}(A\modcat)).$$
	\end{Lem}
	\begin{proof} By
		Lemma \ref{lem-itdist-c}(2), it follows from $A\modcat=\Omega_{\mathscr{D}}^s(\mathcal{D}^{[s,t]}(A\modcat))$ that $$\ITdist(A\modcat)=\ITdist(\mathcal{D}^{[s,t]}(A\modcat)).$$ Then the statement follows from Lemma \ref{lem-it-equi-1}.
	\end{proof}
	
	Let $\mathscr{D}$ be a triangulated category. Let $\mathcal{X}$ and $\mathcal{Y}$ be two subcategories of $\mathscr{D}$. We denoted by
	$\mathcal{X}\star\mathcal{Y}$ the full subcategory of $\mathscr{D}$ consisting of objects $M$ such that there is a triangle $X\ra M\ra Y\ra X[1]$ with $X\in \mathcal{X}$ and $Y\in \mathcal{Y}$. Then $(\mathcal{X}\star\mathcal{Y})\star\mathcal{Z}=\mathcal{X}\star(\mathcal{Y}\star\mathcal{Z})$ for any subclasses $\mathcal{X}$, $\mathcal{Y}$ and $\mathcal{Z}$ of $\mathscr{D}$ by the octahedral axiom. Write
	$$\mathcal{X}_0=0,\; \mathcal{X}_{n}:=\mathcal{X}_{n-1}\star\mathcal{X} \mbox{ for any } n\ge 1.$$
	
	\begin{Lem}\label{lem-r-here-equi}
		Let $A$ be an Artin algebra and $\mathcal{C}\subseteq \Db{A\modcat}$. Let $m\in \mathbb{N}$ and $n\in \mathbb{Z}$. Then
		
		{\rm (1)} $\mathcal{C}$ is relatively $m$-hereditary if and only if there is a complex
		$\cpx{V}\in \Db{A\modcat}$ such that $\mathcal{C}\subseteq \big(\add(\cpx{V})\big)_{m+1}$.
		
		{\rm (2)} $\mathcal{C}$ is $(m,n)$-Igusa-Todorov class if and only if there is a complex $\cpx{V}\in \Db{A\modcat}$ such that $\Omega_{\mathscr{D}}^n(\mathcal{C})\subseteq \big(\add(\cpx{V})\big)_{m+1}$.
	\end{Lem}
	\begin{proof}
		Clearly, (2) follows (1). We prove the statement in (1). Suppose that $\mathcal{C}$ is relatively $m$-hereditary. By Definition \ref{def-r-hered}, there is a complex
		$\cpx{U}\in \Db{A\modcat}$ such that for any $\cpx{X}\in \mathcal{C}$ there exist triangles
		$$\xymatrix@R=.2cm@C=.4cm{
			\cpx{X}_1\ar[r]& \cpx{U}_0\ar[r]& \cpx{X}\ar[r]& \cpx{X}_1[1],\\
			\cpx{X}_2\ar[r]& \cpx{U}_1\ar[r]& \cpx{X}_1\ar[r]& \cpx{X}_2[1],\\
			\cpx{X}_3\ar[r]& \cpx{U}_2\ar[r]& \cpx{X}_2\ar[r]& \cpx{X}_3[1],\\
			& \vdots& \vdots& \\
			\cpx{X}_m\ar[r]& \cpx{U}_{m-1}\ar[r]& \cpx{X}_{m-1}\ar[r]& \cpx{X}_m[1],
		}$$
		where $\cpx{U}_i\in \add(\cpx{U})$ for $0\le i\le m-1$ and $\cpx{X}_m\in\add(\cpx{U})$. Then $$\cpx{X}\in \add(\cpx{U})\star\add(\cpx{U}[1])\star\add(\cpx{U}[2])\star\cdots \star\add(\cpx{U}[m]).$$
		Define $\cpx{V}:=\cpx{U}\oplus \cpx{U}[1] \oplus \cdots\oplus \cpx{U}[m]$. Then $\cpx{X}\in \big(\add(\cpx{U})\big)_{m+1}$.
		
		Conversely, suppose that there is a complex
		$\cpx{V}\in \Db{A\modcat}$ such that $\mathcal{C}\subseteq \big(\add(\cpx{V})\big)_{m+1}$. Then for any $\cpx{X}\in \mathcal{C}$ there exist triangles
		$$\xymatrix@R=.2cm@C=.4cm{
			\cpx{V}_0\ar[r]& \cpx{X}\ar[r]& \cpx{Y}_1\ar[r]& \cpx{V}_0[1],\\
			\cpx{V}_1\ar[r]& \cpx{Y}_1\ar[r]& \cpx{Y}_2\ar[r]& \cpx{V}_1[1],\\
			\cpx{V}_2\ar[r]& \cpx{Y}_2\ar[r]& \cpx{Y}_3\ar[r]& \cpx{V}_2[1],\\
			& \vdots& \vdots& \\
			\cpx{V}_{m-1}\ar[r]& \cpx{Y}_{m-1}\ar[r]& \cpx{Y}_m\ar[r]& \cpx{V}_{m-1}[1],
		}$$
		where $\cpx{V}_i\in \add(\cpx{V})$ for $0\le i\le m-1$ and $\cpx{Y}_m\in\add(\cpx{V})$. Furthermore, we have triangles
		$$\xymatrix@R=.2cm@C=.4cm{
			\cpx{Y}_1[-1]\ar[r]&\cpx{V}_0\ar[r]& \cpx{X}\ar[r]& \cpx{Y}_1,\\
			\cpx{Y}_2[-2]\ar[r]&\cpx{V}_1[-1]\ar[r]& \cpx{Y}_1[-1]\ar[r]& \cpx{Y}_2[-1],\\
			\cpx{Y}_3[-3]\ar[r]&\cpx{V}_2[-2]\ar[r]& \cpx{Y}_2[-2]\ar[r]& \cpx{Y}_3[-2],\\
			& \vdots& \vdots& \\
			\cpx{Y}_m[-m]\ar[r]&\cpx{V}_{m-1}[1-m]\ar[r]& \cpx{Y}_{m-1}[1-m]\ar[r]& \cpx{Y}_m[1-m],
		}$$
		Define $\cpx{U}:=\cpx{V}\oplus \cpx{V}[-1]\oplus \cdots\oplus \cpx{V}[-m]$. By definition \ref{def-r-hered}, $\mathcal{C}$ is relatively $m$-hereditary.
	\end{proof}
	
	\begin{Lem}\label{lem-sumand-proj}
		Let $A$ be an Artin algebra and $\mathcal{C}\subseteq \Db{A\modcat}$. Set $\mathcal{D}=\{\cpx{X}\oplus P\mid \cpx{X}\in \mathcal{C} \text{ and } P\in \pmodcat{A}\}$. Let $m\in \mathbb{N}$ and $n\in \mathbb{Z}$. Then
		
		{\rm (1)} $\mathcal{C}$ is relatively $m$-hereditary if and only if $\mathcal{D}$ is relatively $m$-hereditary.
		
		{\rm (2)} $\mathcal{C}$ is $(m,n)$-Igusa-Todorov class if and only if $\mathcal{D}$ is $(m,n)$-Igusa-Todorov class.
	\end{Lem}
	\begin{proof}
		Clearly, (2) follows (1). We prove the statement in (1). Suppose that $\mathcal{C}$ is relatively $m$-hereditary. By Definition \ref{def-r-hered}, there is a complex
		$\cpx{U}\in \Db{A\modcat}$ such that for any $\cpx{X}\in \mathcal{C}$ there exist triangles
		\begin{gather}
			\cpx{X}_1\lra \cpx{U}_0\lra \cpx{X}\lra \cpx{X}_1[1],\label{triangle-lem-1}\\
			\cpx{X}_2\lra \cpx{U}_1\lra \cpx{X}_1\lra \cpx{X}_2[1],\label{triangle-lem-2}\\
			\cpx{X}_3\lra \cpx{U}_2\lra \cpx{X}_2\lra \cpx{X}_3[1],\label{triangle-lem-3}\\
			\vdots\;\quad\qquad\vdots\;\;\;\nonumber \\
			\cpx{X}_m\lra \cpx{U}_{m-1}\lra \cpx{X}_{m-1}\lra \cpx{X}_m[1],\label{triangle-lem-4}
		\end{gather}
		where $\cpx{U}_i\in \add(\cpx{U})$ for $0\le i\le m-1$ and $\cpx{X}_m\in\add(\cpx{U})$.
		For $P\in \pmodcat{A}$, by the triangle (\ref{triangle-lem-1}), we get a new triangle
		\begin{gather}
			\cpx{X}_1\lra \cpx{U}_0\oplus P\lra \cpx{X}\oplus P\lra \cpx{X}_1[1].\label{triangle-lem-5}
		\end{gather}
		Then $\mathcal{D}$ is relatively $m$-hereditary by Definition \ref{def-r-hered}. On the other hand, due to $\mathcal{C}\subseteq \mathcal{D}$, if $\mathcal{D}$ is relatively $m$-hereditary, then $\mathcal{C}$ is relatively $m$-hereditary. This finishes the proof.
	\end{proof}
	
	{\bf Assumption.} By the above lemma, next, for each subcategory of $\mathcal{C}$ of $\Db{A\modcat}$, we assume $\pmodcat{A}\subseteq \mathcal{C}$ and $\pmodcat{A}\subseteq \Omega_{\mathscr{D}}^n(\mathcal{C})$ for each $n\in \mathbb{Z}$.
	\begin{Lem}\label{lem-itdist-2}
		Let $A$ be an Artin algebra, and $\mathcal{X}, \mathcal{Y}\subseteq \Db{A\modcat}$.
		Then
		$$\max\{\ITdist(\mathcal{X}),\ITdist(\mathcal{Y})\}\le \ITdist(\mathcal{X}\star\mathcal{Y})\le \ITdist(\mathcal{X})+\ITdist(\mathcal{Y})+1.$$
	\end{Lem}
	\begin{proof}
		Let $\cpx{M}\in \mathcal{X}\star\mathcal{Y}$. Then there are $\cpx{X}\in \mathcal{X}$ and $\cpx{Y}\in \mathcal{Y}$ such that we have a triangle \begin{gather}\label{lem-itdist-2-triangle}
			\cpx{X}\lra \cpx{M}\lra \cpx{Y}\lra \cpx{X}[1].
		\end{gather}
		Set $\ITdist(\mathcal{X})=m_1$ and $\ITdist(\mathcal{Y})=m_2$. By Definition \ref{def-itdist-c}, there are integers $n_1$ and $n_2$ such that $\mathcal{X}$ (respectively, $\mathcal{Y}$) is $(m_1,n_1)$-Igusa-Todorov class (respectively, $(m_2,n_2)$-Igusa-Todorov class). Let $n=\max\{n_1,n_2\}$.
		By Lemma \ref{lem-r-here-equi}(2), there are complexes $\cpx{V}_1$ and $\cpx{V}_2$ such that $$
		\Omega_{\mathscr{D}}^{n_1}(\mathcal{X})\subseteq \big(\add(\cpx{V}_1)\big)_{m_1+1}
		\text{ and }
		\Omega_{\mathscr{D}}^{n_2}(\mathcal{Y})\subseteq
		\big(\add(\cpx{V}_2)\big)_{m_2+1}.$$
		By Lemma \ref{lem-prop-omega}(4), we get
		$$\Omega_{\mathscr{D}}^{n}(\mathcal{X})\subseteq \big(\add(\Omega_{\mathscr{D}}^{n-n_1}(\cpx{V}_1))\big)_{m_1+1}
		\text{ and }
		\Omega_{\mathscr{D}}^{n}(\mathcal{Y})\subseteq
		\big(\add(\Omega_{\mathscr{D}}^{n-n_2}(\cpx{V}_2))\big)_{m_2+1}.$$
		On the other hand,
		by Lemma \ref{lem-omega-shift}, applying $\Omega^n_{\mathscr{D}}$ to the triangle (\ref{lem-itdist-2-triangle}), we get a triangle $$\Omega^n_{\mathscr{D}}(\cpx{X})\lra \Omega^n_{\mathscr{D}}(\cpx{M})\oplus P\lra \Omega^n_{\mathscr{D}}(\cpx{Y})\lra \Omega^n_{\mathscr{D}}(\cpx{X})[1],$$
		for some projective $A$-module $P$.
		Then
		\begin{align*}
			\Omega^n_{\mathscr{D}}(\cpx{M})\oplus P
			&\in \big(\add(\Omega_{\mathscr{D}}^{n-n_1}(\cpx{V}_1))\big)_{m_1+1}
			\star
			\big(\add(\Omega_{\mathscr{D}}^{n-n_2}(\cpx{V}_2))\big)_{m_2+1}\\
			&\subseteq \big(\add(\Omega_{\mathscr{D}}^{n-n_1}(\cpx{V}_1)\oplus \Omega_{\mathscr{D}}^{n-n_2}(\cpx{V}_2))\big)_{m_1+m_2+1}.
		\end{align*}
		By Lemma \ref{lem-r-here-equi}(2), $\mathcal{X}\star\mathcal{Y}$ is $(m_1+m_2+1,n)$-Igusa-Todorov class. Thus $$\ITdist(\mathcal{X}\star\mathcal{Y})\le \ITdist(\mathcal{X})+\ITdist(\mathcal{Y})+1.$$
		On the other hand, it follows from $\mathcal{X}\subseteq \mathcal{X}\star\mathcal{Y}$ and $\mathcal{Y}\subseteq \mathcal{X}\star\mathcal{Y}$ that $\ITdist(\mathcal{X})\le \ITdist(\mathcal{X}\star\mathcal{Y})$ and $\ITdist(\mathcal{Y})\le \ITdist(\mathcal{X}\star\mathcal{Y})$ by Lemma \ref{lem-itdist-c}.
		This finishes the proof.
	\end{proof}
	
	\begin{Lem}\label{lem-functor-proj}
		{\rm (\cite[Lemma 3.1]{ww22})}
		Let $F:\Kb{\pmodcat{A}}\to \Kb{\pmodcat{B}}$ be a triangle functor. Then there exist two integers $s<t$ such that  $F(\mathscr{K}^{[p,q]}(\pmodcat{A}))\subseteq \mathscr{K}^{[s+p,t+q]}(\pmodcat{B})$, for any integers $p<q$.
	\end{Lem}
	\begin{proof}
		For the convenience of the reader, we include here a proof. Suppose that $F(A)\in \mathscr{K}^{[s,t]}(\pmodcat{B})$. Take any $\cpx{X}\in \mathscr{K}^{[p,q]}(\pmodcat{A})$. By Lemma \ref{lem-x-tri}, there are triangles
		$$\Omega_{\mathscr{D}}^{i+1}(\cpx{X})\lra X^{-i}\lra  \Omega_{\mathscr{D}}^{i}(\cpx{X})\lra \Omega_{\mathscr{D}}^{i+1}(\cpx{X})[1]$$
		with $X^{-i}\in \pmodcat{A}$ for $-p-1\le i\le -q$. Since $\cpx{X}\in \mathscr{K}^{[p,q]}(\pmodcat{A})$, we see that $\Omega_{\mathscr{D}}^{-p}(\cpx{X})$ is (isomorphic to) a projective $A$-module and $\cpx{X}\simeq \Omega_{\mathscr{D}}^{-q}(\cpx{X})[-q]$ by Lemma \ref{lem-prop-omega}(2). Applying the functor $F$ to there triangles
		$$F(\Omega_{\mathscr{D}}^{i+1}(\cpx{X}))\lra F(X^{-i})\lra  F(\Omega_{\mathscr{D}}^{i}(\cpx{X}))\lra F(\Omega_{\mathscr{D}}^{i+1}(\cpx{X}))[1]$$
		for $-p-1\le i\le -q$. Note that $F(\Omega_{\mathscr{D}}^{-p}(\cpx{X})),F(X^{-i})\in \add(F(A))\subseteq \mathscr{K}^{[s,t]}(\pmodcat{B})$ for $-p-1\le i\le -q$. By using the construction of cones, we get $F(\Omega_{\mathscr{D}}^{-q}(\cpx{X}))\in \mathscr{K}^{[s+p-q,t]}(\pmodcat{B})$. Consequently, we see that $F(\cpx{X})\simeq F(\Omega_{\mathscr{D}}^{-q}(\cpx{X}))[-q]\in \mathscr{K}^{[s+p,t+q]}(\pmodcat{B})$.
	\end{proof}
	
	\begin{Lem}\label{lem-functor}
		{\rm (\cite[Lemma 3.2]{ww22})}
		Let $F:\Db{A\modcat}\to \Db{B\modcat}$ be a triangle functor. Then there exist two integers $s<t$ such that  $F(\mathscr{D}^{[p,q]}(A\modcat))\subseteq \mathscr{D}^{[s+p,t+q]}(B\modcat)$, for any integers $p<q$.
	\end{Lem}
	\begin{proof}
		For the convenience of the reader, we include here a proof. Suppose that $F(S)\in \mathscr{D}^{[p,q]}(B\modcat)$ for any simple $A$-module $S$. Since $A$ bis an Artin algebra, for each $M\in A\modcat$, $M$ has finite composition series:
		$$0=M_0\subseteq M_1\subseteq M_2\subseteq \cdots \subseteq M_r=M.$$
		Then there are triangles $M_i\ra M_{i+1}\ra M_{i+1}/M_i\ra M_i[1]$ for $0\le i\le r-1$. Here each $M_{i+1}/M_i$ is a simple $A$-module. Applying the functor $F$ to these triangles, and then by the cohomological functor $\mathbb{R}\Hom_{\Db{B\modcat}}(B,-)$ from $\Db{B\modcat}$ to the category of ablian groups to these triangles, we get $F(M)\in \mathscr{D}^{[p,q]}(B\modcat)$. Then similar argument as in Lemma \ref{lem-functor-proj}, this result follows.
	\end{proof}
	
	\begin{Lem}\label{lemma-functor-it}
		Let $\mathcal{C}\subseteq \Db{A}$. Assume that $F:\Db{A}\to \Db{B}$ restricts to  $\Kb{\rm proj}$.
		Then $$\ITdist(F(\mathcal{C}))\le \ITdist(\mathcal{C}).$$
	\end{Lem}
	\begin{proof}
		Set $\ITdist(\mathcal{C})=m$.
		By Definition \ref{def-itdist-c}, $\mathcal{C}$ is $(m,n)$-Igusa-Todorov class for some integer $n$. For each $\cpx{Y}\in F(\mathcal{C})$, there is $\cpx{X}\in \mathcal{C}$ such that $\cpx{Y}\simeq F(\cpx{X})$. By Definition \ref{def-r-hered}, there is a complex
		$\cpx{U}\in \Db{A\modcat}$ and triangles in $\Db{A\modcat}$:
		$$\xymatrix@R=.2cm@C=.4cm{
			\cpx{X}_1\ar[r]& \cpx{U}_0\ar[r]& \Omega^n_{\mathscr{D}}(\cpx{X})\ar[r]& \cpx{X}_1[1],\\
			\cpx{X}_2\ar[r]& \cpx{U}_1\ar[r]& \cpx{X}_1\ar[r]& \cpx{X}_2[1],\\
			\cpx{X}_3\ar[r]& \cpx{U}_2\ar[r]& \cpx{X}_2\ar[r]& \cpx{X}_3[1],\\
			& \vdots& \vdots& \\
			\cpx{X}_m\ar[r]& \cpx{U}_{m-1}\ar[r]& \cpx{X}_{m-1}\ar[r]& \cpx{X}_m[1],
		}$$
		where $\cpx{U}_i\in \add(\cpx{U})$ for $0\le i\le m-1$ and $\cpx{X}_m\in\add(\cpx{U})$. Then we have triangles in $\Db{B\modcat}$:
		\begin{equation}\label{f-triangle-1}
			\begin{aligned}
				\xymatrix@R=.2cm@C=.4cm{
					F(\cpx{X}_1)\ar[r]& F(\cpx{U}_0)\ar[r]& F(\Omega^n_{\mathscr{D}}(\cpx{X}))\ar[r]& F(\cpx{X}_1)[1],\\
					F(\cpx{X}_2)\ar[r]& F(\cpx{U}_1)\ar[r]& F(\cpx{X}_1)\ar[r]& F(\cpx{X}_2)[1],\\
					F(\cpx{X}_3)\ar[r]& F(\cpx{U}_2)\ar[r]& F(\cpx{X}_2)\ar[r]& F(\cpx{X}_3)[1],\\
					& \vdots& \vdots& \\
					F(\cpx{X}_m)\ar[r]& F(\cpx{U}_{m-1})\ar[r]& F(\cpx{X}_{m-1})\ar[r]& F(\cpx{X}_m)[1],
				}
			\end{aligned}
		\end{equation}
		where $F(\cpx{U}_i)\in \add(F(\cpx{U}))$ for $0\le i\le m-1$ and $F(\cpx{X}_m)\in\add(F(\cpx{U}))$.
		
		On the other hand, assume that $\cpx{X}\in \mathscr{D}^{[s,t]}(A\modcat)$ such that $t>\max\{-n,s\}$. By Lemma \ref{lem-x-tri}, we have a triangle
		$$\Omega_{\mathscr{D}}^{n}(\cpx{X})[t+n-1]\lra \cpx{\bar{P}}\lra \Omega_{\mathscr{D}}^{-t}(\cpx{X})\lra \Omega_{\mathscr{D}}^{n}(\cpx{X})[t+n],$$ where $\cpx{\bar{P}}\in \mathscr{K}^{[-t-n+1,0]}(\pmodcat{A})$.
		Then we have a triangle
		\begin{align}\label{f-shift}
			\Omega_{\mathscr{D}}^{n}(\cpx{X})[n-1]\lra \cpx{\bar{P}}[-t]\lra \Omega_{\mathscr{D}}^{-t}(\cpx{X})[-t]\lra \Omega_{\mathscr{D}}^{n}(\cpx{X})[n].
		\end{align}
		Define $\cpx{P}:=\cpx{\bar{P}}[-t]$. Then $\cpx{P}\in\mathscr{K}^{[-n+1,t]}(\pmodcat{A})$. By Lemma \ref{lem-prop-omega}(2), $\cpx{X}\simeq \Omega_{\mathscr{D}}^{-t}(\cpx{X})[-t]$. Thus the triangle (\ref{f-shift}) can be written as
		\begin{align}\label{f-shift-1}
			\Omega_{\mathscr{D}}^{n}(\cpx{X})[n-1]\lra \cpx{P}\lra \cpx{X}\lra \Omega_{\mathscr{D}}^{n}(\cpx{X})[n].
		\end{align}
		Applying the functor $F$ to the triangle (\ref{f-shift-1}), we have
		\begin{align}\label{f-shift-2}
			F(\Omega_{\mathscr{D}}^{n}(\cpx{X}))[n-1]\lra F(\cpx{P})\lra F(\cpx{X})\lra F(\Omega_{\mathscr{D}}^{n}(\cpx{X}))[n].
		\end{align}
		Assume that $F(A)\in \mathscr{K}^{[u,v]}(\pmodcat{B})$. By Lemma \ref{lem-functor-proj}, we have $F(\cpx{P})\in \mathscr{K}^{[u-n+1,v+t]}(\pmodcat{B})$. Then $\Omega_{\mathscr{D}}^{i}(F(\cpx{P}))=0$ for $i>n-u-1$.

		By Lemma \ref{lem-omega-shift}, applying $\Omega_{\mathscr{D}}^{n-u+1}$ to the triangle
		\begin{align*}
			F(\cpx{X})\lra F(\Omega_{\mathscr{D}}^{n}(\cpx{X}))[n]\lra F(\cpx{P})[1]\lra F(\cpx{X})[1],
		\end{align*}
		we have a triangle
		\begin{align*}
			\Omega_{\mathscr{D}}^{n-u+1}(F(\cpx{X}))
			\lra \Omega_{\mathscr{D}}^{n-u+1}(F(\Omega_{\mathscr{D}}^{n}(\cpx{X}))[n])\oplus Q
			\lra \Omega_{\mathscr{D}}^{n-u+1}(F(\cpx{P})[1])\lra \Omega_{\mathscr{D}}^{n-u+1}(F(\cpx{X})),
		\end{align*}
		with $Q\in \pmodcat{A}$.
		By Lemma \ref{lem-prop-omega}(3), we get that $\Omega_{\mathscr{D}}^{n-u+1}(F(\cpx{P})[1])\simeq \Omega_{\mathscr{D}}^{n-u}(F(\cpx{P}))=0$.
		Then $$\Omega_{\mathscr{D}}^{n-u+1}(F(\cpx{X}))
		\simeq \Omega_{\mathscr{D}}^{n-u+1}(F(\Omega_{\mathscr{D}}^{n}(\cpx{X}))[n])\oplus Q
		\simeq \Omega_{\mathscr{D}}^{-u+1}(F(\Omega_{\mathscr{D}}^{n}(\cpx{X})))\oplus Q,$$ where the second isomorphism from Lemma \ref{lem-prop-omega}(3). Applying $\Omega_{\mathscr{D}}^{-u+1}$ to the triangles (\ref{f-triangle-1}), by Lemma \ref{lem-omega-shift}, we have triangles
		\begin{equation}\label{f-triangle-2}
			\begin{aligned}
				\xymatrix@R=.2cm@C=.4cm{
					\Omega_{\mathscr{D}}^{-u+1}(F(\cpx{X}_1))\ar[r]&
					\Omega_{\mathscr{D}}^{-u+1}(F(\cpx{U}_0))\oplus Q_0\ar[r]& \Omega_{\mathscr{D}}^{-u+1}(F(\Omega^n_{\mathscr{D}}(\cpx{X})))
					\ar[r]&
					\Omega_{\mathscr{D}}^{-u+1}(F(\cpx{X}_1))[1],\\
					\Omega_{\mathscr{D}}^{-u+1}(F(\cpx{X}_2))
					\ar[r]&
					\Omega_{\mathscr{D}}^{-u+1}(F(\cpx{U}_1))\oplus Q_1
					\ar[r]&
					\Omega_{\mathscr{D}}^{-u+1}(F(\cpx{X}_1))
					\ar[r]&
					\Omega_{\mathscr{D}}^{-u+1}(F(\cpx{X}_2))[1],\\
					\Omega_{\mathscr{D}}^{-u+1}(F(\cpx{X}_3))
					\ar[r]&
					\Omega_{\mathscr{D}}^{-u+1}(F(\cpx{U}_2))\oplus Q_2
					\ar[r]&
					\Omega_{\mathscr{D}}^{-u+1}(F(\cpx{X}_2))
					\ar[r]&
					\Omega_{\mathscr{D}}^{-u+1}(F(\cpx{X}_3))[1]
					,\\
					& \vdots& \vdots& \\
					\Omega_{\mathscr{D}}^{-u+1}(F(\cpx{X}_m))
					\ar[r]&
					\Omega_{\mathscr{D}}^{-u+1}(F(\cpx{U}_{m-1}))\oplus Q_{m-1}
					\ar[r]&
					\Omega_{\mathscr{D}}^{-u+1}(F(\cpx{X}_{m-1}))
					\ar[r]&
					\Omega_{\mathscr{D}}^{-u+1}(F(\cpx{X}_m))[1],
				}
			\end{aligned}
		\end{equation}
		where $Q_i\in \pmodcat{A}$ for $0\le i\le m-1$.
		It follows from $\Omega_{\mathscr{D}}^{n-u+1}(F(\cpx{X}))
		\simeq \Omega_{\mathscr{D}}^{-u+1}(F(\Omega_{\mathscr{D}}^{n}(\cpx{X})))\oplus Q$ and the triangles (\ref{f-triangle-2}) that $\Omega_{\mathscr{D}}^{n-u+1}(F(\mathcal{C}))$ is relatively $m$-hereditary. Thus $\ITdist(F(\mathcal{C}))\le m=\ITdist(\mathcal{C})$.
	\end{proof}

	\begin{Theo}\label{thm-rec}
		Let $A$, $B$ and $C$ be Artin algebras.
		Suppose that there is a recollement among the derived categories
		$\D{A\Modcat}$, $\D{B\Modcat}$ and $\D{C\Modcat}$:
		\begin{align}\label{sect-rec}
			\xymatrixcolsep{4pc}\xymatrix{
				\D{B\Modcat} \ar[r]|{i_*=i_!} &\D{A\Modcat} \ar@<-2ex>[l]|{i^*} \ar@<2ex>[l]|{i^!} \ar[r]|{j^!=j^*}  &\D{C\Modcat}. \ar@<-2ex>[l]|{j_!} \ar@<2ex>[l]|{j_{*}}
			}
		\end{align}
		
		{\rm (1)(i)} If the recollement {\rm (\ref{sect-rec})} extends one step downwards, then $\ITdist(C)\le \ITdist(A)$.
		
		{\rm (ii)} If the recollement {\rm (\ref{sect-rec})} extends one step upwards, then $\ITdist(B)\le \ITdist(A)$.
		
		{\rm (2)} If the recollement {\rm (\ref{sect-rec})} extends one step downwards and extends one step upwards, then $$\max\{\ITdist(B),\ITdist(C)\}\le \ITdist(A)\le \ITdist(B)+\ITdist(C)+1.$$
		
		{\rm (3)(i)} If the recollement {\rm (\ref{sect-rec})} extends one step downwards and $\gd(B)<\infty$, then $\ITdist(C)= \ITdist(A)$.
		
		{\rm (ii)} If $\gd(C)<\infty$, then $\ITdist(B)= \ITdist(A)$.
	\end{Theo}
	
	\begin{proof}
		(1) (i) By Lemma \ref{lem-rec-prop}(1), $j^{*}$ and $j_{*}$ restrict to $\Db{\rm mod}$, and $j_{*}$ is fully faithful. Thus $C\modcat=j^{*}j_{*}(C\modcat)$. Also $j^{*}$ restricts to $\Kb{\rm proj}$. By Lemma \ref{lemma-functor-it}, we have $\ITdist(C\modcat)\le \ITdist(j_{*}(C\modcat))$. Since $j_{*}$ restricts to $\Db{\rm mod}$,
		there are integers $s<t$ such that $j_{*}(C\modcat)\subseteq \mathscr{D}^{[s,t]} (A\modcat)$ by Lemma \ref{lem-functor}.
		Then $\ITdist(j_{*}(C\modcat))\le\ITdist(\mathscr{D}^{[s,t]} (A\modcat))$.
		By Lemma \ref{lem-it-equi}, $\ITdist(\mathscr{D}^{[s,t]}(A\modcat))=\ITdist(A)$.
		Thus $\ITdist(C\modcat)\le \ITdist(A)$.
		
		(ii) By Lemma \ref{lem-rec-prop}(2), $i^*$ has a left adjoint $F$, $j_{!}$ has a left adjoint $G$, and we have a recollement
		\begin{align*}
			\xymatrixcolsep{4pc}\xymatrix{
				\D{C\Modcat} \ar[r]|{j_!} &\D{A\Modcat} \ar@<-2ex>[l]|{G} \ar@<2ex>[l]|{j^!} \ar[r]|{i^*}  &\D{B\Modcat}, \ar@<-2ex>[l]|{F} \ar@<2ex>[l]|{i_{*}}
			}
		\end{align*}
		which extends one step downwards. By (1), we obtain $\ITdist(B)\le \ITdist(A)$.
		
		(2) By (1),  we have to show $\ITdist(A)\le \ITdist(B)+\ITdist(C)+1$. Indeed, by Lemma \ref{lem-rec-prop}, we have two recollements:
		\begin{align*}
			\xymatrixcolsep{4pc}\xymatrix{
				\Kb{\pmodcat{B}} \ar[r]|{i_*=i_!} &\Kb{\pmodcat{A}} \ar@<-2ex>[l]|{i^*} \ar@<2ex>[l]|{i^!} \ar[r]|{j^!=j^*}  &\Kb{\pmodcat{C}}, \ar@<-2ex>[l]|{j_!} \ar@<2ex>[l]|{j_{*}}
			}
		\end{align*}
		\begin{align*}
			\xymatrixcolsep{4pc}\xymatrix{
				\Db{B\modcat} \ar[r]|{i_*=i_!} &\Db{A\modcat} \ar@<-2ex>[l]|{i^*} \ar@<2ex>[l]|{i^!} \ar[r]|{j^!=j^*}  &\Db{C\modcat}. \ar@<-2ex>[l]|{j_!} \ar@<2ex>[l]|{j_{*}}
			}
		\end{align*}
		Then for each $A$-module $X$, there is a triangle $j_{!}j^{!}(X)\ra X \ra i_{*}i^{*}(X)\ra j_{!}j^{!}(X)$ in $\Db{A\modcat}$. Thus $$A\modcat\subseteq j_{!}j^{!}(A\modcat)\star i_{*}i^{*}(A\modcat).$$
		By Lemma \ref{lem-itdist-c} and Lemma \ref{lem-itdist-2},  $$\ITdist(A\modcat)\le \ITdist(j_{!}j^{!}(A\modcat))+\ITdist(i_{*}i^{*}(A\modcat))+1.$$
		By a similar augument as in (1), we know that $\ITdist(j_{!}j^{!}(A\modcat))\le \ITdist(C\modcat)$ and $\ITdist(i_{*}i^{*}(A\modcat))\le \ITdist(B\modcat)$. Then the desired result follows from Lemma \ref{lem-it-equi-1}.
		
		(3)(i) By (1)(i), we obtain $\ITdist(C)\le \ITdist(A)$. We have to show $\ITdist(C)\ge \ITdist(A)$. Indeed, by Lemma \ref{lem-rec-prop}, we have a right recollement
		\begin{align*}
			\xymatrixcolsep{4pc}\xymatrix{
				\Db{B\modcat} \ar[r]|{i_*=i_!} &\Db{A\modcat}  \ar@<2ex>[l]|{i^!} \ar[r]|{j^!=j^*}  &\Db{C\modcat}. \ar@<2ex>[l]|{j_{*}}
			}
		\end{align*} Thus for each $X\in A\modcat$, there is a triangle in $\Db{A\modcat}$
		$$i_{!}i^{!}(X)\lra X\lra j_{*}j^{*}(X)\lra i_{*}i^{!}(X)[1].$$
		By Lemma \ref{lem-functor}, there exist two integers $p<q$ such that $i^{!}(A\modcat)\subseteq \mathscr{D}^{[p,q]}(B\modcat)$. Due to $u:=\gd(B)<\infty$, $\mathscr{D}^{[p,q]}(B\modcat)\subseteq \mathscr{K}^{[p-u,q]}(\pmodcat{B})$.
		Note that $i_{!}$ restricts to $\Kb{\rm proj}$. By Lemma \ref{lem-functor-proj}, there exist two integers $s<t$ such that  $i_{!}i^{!}(A\modcat)\subseteq \mathscr{K}^{[s+p-u,t+q]}(\pmodcat{A})$. Then $\Omega_{\mathscr{D}}^l(i_{!}i^{!}(X))=0$ for $l>u-s-p$.
		By Lemma \ref{lem-omega-shift}, we have $\Omega_{\mathscr{D}}^l(X)\oplus P \simeq \Omega_{\mathscr{D}}^l(j_{*}j^{*}(X))$ with $P\in \pmodcat{A}$ for $l>u-s-p$. Thus
		$\ITdist(\Omega^l_{\mathscr{D}}(A\modcat))=\ITdist(\Omega^l_{\mathscr{D}}(j_{*}j^{*}(A\modcat)))$ for $l>u-s-p$. Take $l=u-s-p+1$, by Lemma \ref{lem-itdist-c}(2), $\ITdist(A\modcat)=\ITdist(j_{*}j^{*}(A\modcat))$. Since $\gd(B)<\infty$ and $i^{!}$ restricts to $\Db{\rm mod}$, $i^{!}$ restricts to $\Kb{\rm proj}$. By Lemma \ref{lem-rec-prop}(2), $j_{*}$ restricts to $\Kb{\rm proj}$. By a similar augument as in (1), we know that $\ITdist(j_{*}j^{*}(A\modcat))\le \ITdist(C\modcat)$. Thus $\ITdist(A\modcat)\le \ITdist(C\modcat)$. Hence $\ITdist(A\modcat)=\ITdist(C\modcat)$. By Lemma \ref{lem-it-equi-1}, $\ITdist(A)=\ITdist(C)$.
		
		(ii) By Lemma \ref{lem-rec-prop}, $j^{*}$ restricts to $\Db{\rm mod}$. Then $j^{*}(A)$ and  $j^{*}(D(A))$ in $\Db{C\modcat}$. It follows from $\gd(C)<\infty$ that $\Db{C\modcat}=\Kb{\pmodcat{C}}=\Kb{\imodcat{C}}$. Thus the recollement (\ref{sect-rec}) extends one step downwards and extends one step upwards by Lemma \ref{lem-rec-prop}. By (1)(ii), we obtain $\ITdist(B)\le \ITdist(A)$.
		
		Now, we prove $\ITdist(B)\ge \ITdist(A)$. Indeed, by a similar augument as in (2), for each $X\in A\modcat$, there is a triangle in $\Db{A\modcat}$ $$j_{!}j^{!}(X)\lra X \lra i_{*}i^{*}(X)\lra j_{!}j^{!}(X)[1].$$
		By Lemma \ref{lem-functor}, since $j^{!}$ restricts to $\Db{\rm mod}$, there exist two integers $p<q$ such that $j^{!}(A\modcat)\subseteq \mathscr{D}^{[p,q]}(C\modcat)$. Due to $v:=\gd(C)<\infty$, $\mathscr{D}^{[p,q]}(C\modcat)\subseteq \mathscr{K}^{[p-v,q]}(\pmodcat{C})$.
		Note that $j_{!}$ restricts to $\Kb{\rm proj}$. By Lemma \ref{lem-functor-proj}, there exist two integers $s<t$ such that  $j_{!}j^{!}(A\modcat)\subseteq \mathscr{K}^{[s+p-v,t+q]}(\pmodcat{A})$. Then $\Omega_{\mathscr{D}}^l(j_{!}j^{!}(X))=0$ for $l>v-s-p$. By Lemma \ref{lem-omega-shift}, we have $\Omega_{\mathscr{D}}^l(X)\oplus Q \simeq \Omega_{\mathscr{D}}^l(i_{*}i^{*}(X))$ with $Q\in \pmodcat{A}$ for $l>v-s-p$. Then
		$\ITdist(\Omega^l_{\mathscr{D}}(A\modcat))=\ITdist(\Omega^l_{\mathscr{D}}(i_{*}i^{*}(A\modcat)))$ for $l>v-s-p$. Take $l=v-s-p+1$, by Lemma \ref{lem-itdist-c}(2), $\ITdist(A\modcat)=\ITdist(i_{*}i^{*}(A\modcat))$.
		By a similar augument as in (1), we know that $\ITdist(i_{*}i^{*}(A\modcat))\le \ITdist(B\modcat)$. Thus $\ITdist(A\modcat)\le \ITdist(B\modcat)$. Hence $\ITdist(A\modcat)=\ITdist(B\modcat)$. By Lemma \ref{lem-it-equi-1}, $\ITdist(A)=\ITdist(B)$.
	\end{proof}
	
	\begin{Koro}\label{cor-ladder}
		Let $A$, $B$ and $C$ be Artin algebras.
		
		{\rm (1)} Suppose that there is a ladder of $\D{A\Modcat}$ by $\D{B\Modcat}$ and $\D{C\Modcat}$ whose height is equal to $3$. Then $$\max\{\ITdist(B),\ITdist(C)\}\le \ITdist(A)\le \ITdist(B)+\ITdist(C)+1.$$
		
		{\rm (2)} Suppose that there is a recollement among the derived categories
		$\Db{A\modcat}$, $\Db{B\modcat}$ and $\Db{C\modcat}$:
		\begin{align*}
			\xymatrixcolsep{4pc}\xymatrix{
				\Db{B\modcat} \ar[r]|{i_*=i_!} &\Db{A\modcat} \ar@<-2ex>[l]|{i^*} \ar@<2ex>[l]|{i^!} \ar[r]|{j^!=j^*}  &\Db{C\modcat}. \ar@<-2ex>[l]|{j_!} \ar@<2ex>[l]|{j_{*}}
			}
		\end{align*}
		Then $$\max\{\ITdist(B),\ITdist(C)\}\le \ITdist(A)\le \ITdist(B)+\ITdist(C)+1.$$
	\end{Koro}
	\begin{proof}
		(1) By the definition of ladder, without loss of generality, we assume that there is a recollement form as the recollement (\ref{sect-rec}) which extends one step downwards and one step upwards. Thus the statement in (1) follows from Theorem \ref{thm-rec}(2).
		
		(2) By \cite[Proposition 4.1]{akly17}, the recollement of bounded derived categories can be lifted to a recollement of unbounded derived categories, where all functors can be restricted to $\Db{\rm mod}$. By \cite[Lemma 3.4]{ww22}, the recollement of unbounded derived categories extends one step downwards and one step upwards. Thus the statement in (2) follows from Theorem \ref{thm-rec}(2).
	\end{proof}
	
	The following corollary tells us that the Igusa-Todorov distance is an invariant under derived equivalences.
	\begin{Koro}\label{cor-der}
		Let $A$ and $B$ be Artin algebras. If $A$ and $B$ are derived-equivalent, then $$\ITdist(A)=\ITdist(B).$$
	\end{Koro}
	
	By Lemma \ref{lem-it} and Lemma \ref{lem-rec-prop}, Theorem \ref{thm-rec} generalizes \cite[Theorem 1.1]{ww22}.
	
	\begin{Koro}\label{thm-rec-cor}
		Let $A$, $B$ and $C$ be Artin algebras.
		Suppose that there is a recollement among the derived categories
		$\D{A\Modcat}$, $\D{B\Modcat}$ and $\D{C\Modcat}:$
		\begin{align}\label{sect-rec-cor}
			\xymatrixcolsep{4pc}\xymatrix{
				\D{B\Modcat} \ar[r]|{i_*=i_!} &\D{A\Modcat} \ar@<-2ex>[l]|{i^*} \ar@<2ex>[l]|{i^!} \ar[r]|{j^!=j^*}  &\D{C\Modcat}. \ar@<-2ex>[l]|{j_!} \ar@<2ex>[l]|{j_{*}}
			}
		\end{align}
		
		{\rm (1)} Suppose that $A$ is syzygy-finite $($respectively, Igusa-Todorov$).$
		
		{\rm (i)} If the recollement {\rm (\ref{sect-rec-cor})} extends one step downwards, then $C$ is syzygy-finite $($respectively, Igusa-Todorov$).$
		
		{\rm (ii)} If the recollement $(\ref{sect-rec-cor})$ extends one step upwards, then $B$ is syzygy-finite $($respectively, Igusa-Todorov$).$
		
		{\rm (2)} Suppose that the recollement $(\ref{sect-rec-cor})$ extends one step downwards and one step upwards. If both $B$ and $C$ are syzygy-finite, then $A$ is Igusa-Todorov.
		
		{\rm (3)(i)} If the recollement $(\ref{sect-rec-cor})$ extends one step downwards and $\gd(B)<\infty$, then $C$ is syzygy-finite $($respectively, Igusa-Todorov$)$ if and only if $A$ is syzygy-finite $($respectively, Igusa-Todorov$).$
		
		{\rm (ii)} If $\gd(C)<\infty$, then $B$ is syzygy-finite $($respectively, Igusa-Todorov$)$ if and only if $A$ is syzygy-finite $($respectively, Igusa-Todorov{\rm )}.
	\end{Koro}
	
	Let $\mathcal{A}$ and $\mathcal{B}$ be abelian categories. Recall that a functor $F:\mathscr{D}(\mathcal{A})\ra \mathscr{D}(\mathcal{B})$ between their derived categories is called an \textit{eventually homological} isomorphism (\cite{qin20}) if there is an integer $s$ such that for every pair of $X,Y\in \mathcal{A}$, and every $p>s$, there is an isomorphism
	$$\Hom_{\mathscr{D}(\mathcal{A})}(X,Y[p])\simeq \Hom_{\mathscr{D}(\mathcal{C})}(F(X),F(Y)[p])$$
	as abelian groups.
	
	The following corollary is a consequence of Theorem \ref{thm-rec}(2) and \cite[Theorem 4.2]{ww22}.
	\begin{Koro}
		Let $A$, $B$ and $C$ be Artin algebras.
		Suppose that there is a recollement among the derived categories
		$\D{A\Modcat}$, $\D{B\Modcat}$ and $\D{C\Modcat}${\rm :}
		\begin{align*}
			\xymatrixcolsep{4pc}\xymatrix{
				\D{B\Modcat} \ar[r]|{i_*=i_!} &\D{A\Modcat} \ar@<-2ex>[l]|{i^*} \ar@<2ex>[l]|{i^!} \ar[r]|{j^!=j^*}  &\D{C\Modcat}. \ar@<-2ex>[l]|{j_!} \ar@<2ex>[l]|{j_{*}}
			}
		\end{align*}
		If $j^{*}$ is an eventually homological isomorphism, then $$\max\{\ITdist(B),\ITdist(C)\}\le \ITdist(A)\le \ITdist(B)+\ITdist(C)+1.$$
	\end{Koro}
	
	Now, we apply our result to stratifying ideals and triangular amtrix algebras.
	
	Let $A$ be an Artin algebra and let $e \in A$ be an idempotent such that $AeA$ is a stratifying ideal, that is, $Ae\otimes^{\mathbb{L}}_{eAe}eA\simeq AeA$. From \cite{CPS}, there is a recollement
	\begin{align}\label{rec-strat}
		\xymatrixcolsep{4pc}\xymatrix{
			\D{A/AeA\Modcat} \ar[r]|{i_*=i_!} &\D{A\Modcat} \ar@<-2ex>[l]|{i^*} \ar@<2ex>[l]|{i^!} \ar[r]|{j^!=j^*}  &\D{eAe\Modcat}, \ar@<-2ex>[l]|{j_!} \ar@<2ex>[l]|{j_{*}}
		}
	\end{align}
	where $i_{*}=A/AeA\otimes_{A/AeA}-$ and $j_{!}=Ae\otimes_{eAe}^{\mathbb{L}}-$. Then $i_{*}(A/AeA)=A/AeA$.
	
	(1) Suppose that $\pd(_AAeA)<\infty$. Then $\pd(_AA/AeA)<\infty$. By Lemma \ref{lem-rec-prop}(1), $i_{*}$ restricts to $\Kb{\rm proj}$ and the recollement (\ref{rec-strat}) extends one step downwards.
	
	(2) Suppose $\pd(Ae_{eAe})<\infty$. Then $\mathbb{R}\Hom_{\D{{eAe\text{-Mod}}}}(\mathbb{R}\Hom_{\D{eAe^{\opp}\Modcat}}(Ae,eAe),eAe)\simeq Ae$ in $\D{eAe\mbox{-mod}}$. Thus $j_{!}=Ae\otimes_{eAe}^{\mathbb{L}}-$ has the left adjoint $\mathbb{R}\Hom_{\D{eAe^{\opp}\Modcat}}(Ae,eAe)\otimes_{A}^{\mathbb{L}}-$. By Lemma \ref{lem-rec-prop}(2), the recollement (\ref{rec-strat}) extends one step upwards.
	
	Applying Theorem \ref{thm-rec}, we have the following corollary.
	
	\begin{Koro}
		Let $A$ be an Artin algebra and let $e \in A$ be an idempotent such that $AeA$ is a stratifying ideal. The the following holds.
		
		{\rm (1)(i)} If $\pd(_AAeA)<\infty$, then $\ITdist(eAe)\le \ITdist(A)$.
		
		{\rm (ii)} If $\pd(Ae_{eAe})<\infty$, then $\ITdist(A/AeA)\le \ITdist(A)$.
		
		{\rm (2)} If $\pd(_AAeA)<\infty$ and $\pd(Ae_{eAe})<\infty$, then $$\max\{\ITdist(A/AeA),\ITdist(eAe)\}\le \ITdist(A)\le \ITdist((A/AeA)+\ITdist(eAe)+1.$$
		
		{\rm (3)(i)} If $\gd(A/AeA)<\infty$ and $\pd(_AAeA)<\infty$, then $\ITdist(eAe)= \ITdist(A)$.
		
		{\rm (ii)} If $\gd(eAe)<\infty$, then $\ITdist(A/AeA)= \ITdist(A)$.
	\end{Koro}
	
	Let $B$ and $C$ be Artin algebras, $_BM_C$ be a $B$-$C$-bimodule. Consider the triangular matrix algebra
	$A=\left(\begin{smallmatrix}
		B&M\\
		0&C
	\end{smallmatrix}\right)$.
	Set $e_1=\left(\begin{smallmatrix}
		1&0\\
		0&0
	\end{smallmatrix}\right)$ and $e_2=\left(\begin{smallmatrix}
		0&0\\
		0&1
	\end{smallmatrix}\right)$. Note that $Ae_2A$ is a statifying ideals. Then we have a recollement
	\begin{align}\label{rec-tri}
		\xymatrixcolsep{4pc}\xymatrix{
			\D{B\Modcat} \ar[r]|{i_*=i_!} &\D{A\Modcat} \ar@<-2ex>[l]|{i^*} \ar@<2ex>[l]|{i^!} \ar[r]|{j^!=j^*}  &\D{C\Modcat}, \ar@<-2ex>[l]|{j_!} \ar@<2ex>[l]|{j_{*}}
		}
	\end{align}
	where $i_{*}=Ae_1\otimes_{B}^{\mathbb{L}}-$, $i^{!}=e_1A\otimes_{A}^{\mathbb{L}}-$, $j_{!}=Ae_2\otimes_{C}^{\mathbb{L}}-$ and $j^!=e_2A\otimes_{A}^{\mathbb{L}}-$. Note that $i_{*}(B)=Ae_1$. Since $Ae_1$ is a projective $A$-module, By Lemma \ref{lem-rec-prop}(1), $i_{*}$ restricts to $\Kb{\rm proj}$ and the recollement (\ref{rec-strat}) extends one step downwards.
	
	(1) Suppose $\pd(_BM)<\infty$. Note that $i^{!}(A)=e_1A\simeq B\oplus M$ as $B$-modules. Thus $\pd(_Bi^{!}(A))<\infty$. By Lemma \ref{lem-rec-prop}(1), $i^{!}$ restricts to $\Kb{\rm proj}$ and the recollement (\ref{rec-strat}) extends two step downwards. So, there is a ladder of $\D{A\Modcat}$ by $\D{B\Modcat}$ and $\D{C\Modcat}$ whose height is equal to $3$.
	
	(2) Suppose $\pd(M_C)<\infty$. Then $j_{!}(A)=Ae_2\simeq C\oplus M$ as right $C$-modules. Thus $\pd_{C^{\opp}}(Ae_2)<\infty$. Note that $\mathbb{R}\Hom_{\D{C\Modcat}}(\mathbb{R}\Hom_{\D{C^{\opp}\Modcat}}(Ae_2,C),C)\simeq Ae_2$ in $\D{C\Modcat}$. Thus $j_{!}=Ae_2\otimes_{C}^{\mathbb{L}}-$ has the left adjoint $\mathbb{R}\Hom_{\D{C^{\opp}\Modcat}}(Ae_2,C)\otimes_{A}^{\mathbb{L}}-$. By Lemma \ref{lem-rec-prop}(2), the recollement (\ref{rec-strat}) extends one step upwards. So, there is a ladder of $\D{A\Modcat}$ by $\D{B\Modcat}$ and $\D{C\Modcat}$ whose height is equal to $3$.
	
	Applying Theorem \ref{thm-rec} and Corollary \ref{cor-ladder}(1), we have the following corollary.
	\begin{Koro}\label{cor-tri}
		Let
		$A=\left(\begin{smallmatrix}
			B&M\\
			0&C
		\end{smallmatrix}\right)$
		be the triangular matrix algebra. The the following holds.
		
		{\rm (1)} $\ITdist(C)\le \ITdist(A)$.
		
		{\rm (2)} If $\pd(M_C)<\infty$ or $\pd(_BM)<\infty$, then $$\max\{\ITdist(B),\ITdist(C)\}\le \ITdist(A)\le \ITdist(B)+\ITdist(C)+1.$$
		
		{\rm (3)(i)} If $\gd(B)<\infty$, then $\ITdist(C)= \ITdist(A)$.
		
		{\rm (ii)} If $\gd(C)<\infty$, then $\ITdist(B)= \ITdist(A)$.
	\end{Koro}
	
	Next, we consider one-point extensions and one-point coextensions. Let $B$ be a finite dimensional algebra over a field $k$ and $_BM$ is a $B$-module. Recall that the one-point extension of $B$
	by the module $_BM$ is the triangular matrix algebra $A=\left(\begin{smallmatrix}
		B&M\\
		0&k
	\end{smallmatrix}\right)$. Dually, the one-point coextension of $B$
	by the module $_BM$ is the triangular matrix algebra $A=\left(\begin{smallmatrix}
		k&D(_BM)\\
		0&B
	\end{smallmatrix}\right)$. Applying Corollary \ref{cor-tri}(3) to both one-point extension and one-point coextension, we get the following result.
	
	\begin{Koro}\label{cor-point}
		Let $B$ be a finite dimensional algebra over a field $k$ and $_BM$ is a $B$-module.
		
		{\rm (1)} Let $A=\left(\begin{smallmatrix}
			B&M\\
			0&k
		\end{smallmatrix}\right)$ be the one-point extension of $B$
		by $_BM$. Then $\ITdist(B)=\ITdist(A)$.
		
		{\rm (2)} Let $A=\left(\begin{smallmatrix}
			k&D(_BM)\\
			0&B
		\end{smallmatrix}\right)$ be the one-point coextension of $B$
		by $_BM$. Then $\ITdist(B)=\ITdist(A)$.
	\end{Koro}
	Finally, we will apply our results to produce some examples of
	algebras whose Igusa-Todorov distances are no less than $2$.
	\begin{Bsp}
		{\rm
			Let $\Lambda$ be finite-dimensional $k$-algebras over a field $k$ given by quiver with relations:
			\begin{equation*}
				\begin{tikzpicture}[scale=1.0, every node/.style={scale=1.0}]
					
					\node[] at (3,1) {$1$};
					\node [above] (adhm11) at (3.1,0.9) {};
					\node [above] (adhm21) at (3.1,0.8) {};
					
					\node [above] (adhm1) at (3,1) {};
					\node [below] (adhm2) at (3,1) {};
					\node [below] (adhm3) at (3,1) {};
					
					\node[] at (5,1.5) {$2$};
					\node [above] (adhm12) at (4.9,1.35) {};
					\node [above] (adhm2p1) at (5.1,1.4) {};
					
					\node[] at (7,1.5) {$p_1$};
					\node [above] (adhmp1p2) at (7.1,1.4) {};
					\node [above] (adhmp12) at (6.9,1.4) {};
					
					\node[] at (9,1.5) {$p_2$};
					\node [above] (adhmp2p1) at (8.9,1.4) {};
					
					\node[] at (10,1.5) {$\cdots$};
					
					\node[] at (11,1.5) {$p_{n-1}$};
					\node [above] (adhmpn1pn) at (11.3,1.4) {};
					
					\node[] at (13,1.5) {$p_{n}$};
					\node [above] (adhmpnpn1) at (12.9,1.4) {};
					
					%
					\node[] at (5,0.5) {$3$};
					\node [above] (adhm13) at (4.9,0.3) {};
					\node [above] (adhm3q1) at (5.1,0.4) {};
					
					\node[] at (7,0.5) {$q_1$};
					\node [above] (adhmq1q2) at (7.1,0.4) {};
					\node [above] (adhmq13) at (6.9,0.4) {};
					
					\node[] at (9,0.5) {$q_2$};
					\node [above] (adhmq2q1) at (8.9,0.4) {};
					
					\node[] at (10,0.5) {$\cdots$};
					
					\node[] at (11,0.5) {$q_{n-1}$};
					\node [above] (adhmqn1qn) at (11.3,0.4) {};
					
					\node[] at (13,0.5) {$q_{n}$};
					\node [above] (adhmqnqn1) at (12.9,0.4) {};
					
					%
					\node[] at (1.5,2.4) {$\Lambda:$};

					\node[] at (16,2.0) {$\alpha_1\alpha_2+\alpha_2\alpha_1=0,$};

					\node[] at (16,1.5) {$\alpha_1\alpha_3+\alpha_3\alpha_1=0,$};

					\node[] at (16,1.0) {$\alpha_2\alpha_3+\alpha_3\alpha_2=0,$};

					\node[] at (16,0.5) {$\alpha_1^2=\alpha_2^2=\alpha_3^2=0,$};

					\node[] at (16,0.0) {$\alpha_1\beta=\alpha_2\beta=\alpha_3\beta=0,$};

					\node[] at (16,-0.5) {$\gamma\alpha_1=\gamma\alpha_2=\gamma\alpha_3=0.$};
					
					\draw [thick, <-] (adhm1) to [out=120,in=30,looseness=20] node[above] {$\alpha_1$}(adhm1);
					\draw [thick, <-] (adhm2) to [out=320,in=230,looseness=20] node[below] {$\alpha_3$}(adhm2);
					\draw [thick, <-] (adhm3) to [out=215,in=125,looseness=20] node[left] {$\alpha_2$}(adhm3);
					
					\draw [thick, <-] (adhm11)  -- node [above] {$\beta$} (adhm12);
					\draw [thick, ->] (adhm21)  -- node [below] {$\gamma$} (adhm13);
					\draw [thick, <-] (adhm2p1)  -- node [above]{} (adhmp12);
					\draw [thick, <-] (adhmp1p2)  -- node [above]{} (adhmp2p1);
					\draw [thick, <-] (adhmpn1pn)  -- node [above]{} (adhmpnpn1);
					\draw [thick, ->] (adhm3q1)  -- node [above]{} (adhmq13);
					\draw [thick, ->] (adhmq1q2)  -- node [above]{} (adhmq2q1);
					\draw [thick, ->] (adhmqn1qn)  -- node [above]{} (adhmqnqn1);
				\end{tikzpicture}
			\end{equation*}
			We claim $\ITdist(\Lambda)=2$. Indeed, let $A$, $B$ and $C$ be finite-dimensional $k$-algebras over a field $k$ given by quivers with relations in Figure 1, Figure 2 and Figure 3, respectively.
			\begin{equation*}
				\begin{tikzpicture}[scale=1.0, every node/.style={scale=1.0}]
					
					\node[] at (3,1) {$1$};
					\node [above] (adhm11) at (3.1,0.9) {};
					\node [above] (adhm21) at (3.1,0.8) {};
					
					\node [above] (adhm1) at (3,1) {};
					\node [below] (adhm2) at (3,1) {};
					\node [below] (adhm3) at (3,1) {};

					\node[] at (5,1.5) {$2$};
					\node [above] (adhm12) at (4.9,1.35) {};
					\node[] at (5,0.5) {$3$};
					\node [above] (adhm13) at (4.9,0.3) {};

					\node[] at (1.5,2.4) {$A:$};

					\node[] at (5.5,-1.5) {Figure 1};

					\node[] at (8,2.0) {$\alpha_1\alpha_2+\alpha_2\alpha_1=0,$};

					\node[] at (8,1.5) {$\alpha_1\alpha_3+\alpha_3\alpha_1=0,$};

					\node[] at (8,1.0) {$\alpha_2\alpha_3+\alpha_3\alpha_2=0,$};

					\node[] at (8,0.5) {$\alpha_1^2=\alpha_2^2=\alpha_3^2=0,$};

					\node[] at (8,0.0) {$\alpha_1\beta=\alpha_2\beta=\alpha_3\beta=0,$};

					\node[] at (8,-0.5) {$\gamma\alpha_1=\gamma\alpha_2=\gamma\alpha_3=0.$};
					
					\draw [thick, <-] (adhm1) to [out=120,in=30,looseness=20] node[above] {$\alpha_1$}(adhm1);
					\draw [thick, <-] (adhm2) to [out=320,in=230,looseness=20] node[below] {$\alpha_3$}(adhm2);
					\draw [thick, <-] (adhm3) to [out=215,in=125,looseness=20] node[left] {$\alpha_2$}(adhm3);
					
					\draw [thick, <-] (adhm11)  -- node [above] {$\beta$} (adhm12);
					\draw [thick, ->] (adhm21)  -- node [below] {$\gamma$} (adhm13);
				\end{tikzpicture}
			\end{equation*}
			\begin{equation*}
				\begin{tikzpicture}[scale=1.0, every node/.style={scale=1.0}]
					
					\node[] at (3,1) {$1$};
					\node [above] (adhm11) at (3.1,0.9) {};
					\node [above] (adhm21) at (3.1,0.8) {};
					
					\node [above] (adhm1) at (3,1) {};
					\node [below] (adhm2) at (3,1) {};
					\node [below] (adhm3) at (3,1) {};

					\node[] at (5,1.5) {$2$};
					\node [above] (adhm12) at (4.9,1.35) {};
					%

					%
					\node[] at (1.5,2.4) {$B:$};
					
					\node[] at (5.5,-1.7) {Figure 2};

					\node[] at (8,2.0) {$\alpha_1\alpha_2+\alpha_2\alpha_1=0,$};

					\node[] at (8,1.5) {$\alpha_1\alpha_3+\alpha_3\alpha_1=0,$};

					\node[] at (8,1.0) {$\alpha_2\alpha_3+\alpha_3\alpha_2=0,$};

					\node[] at (8,0.5) {$\alpha_1^2=\alpha_2^2=\alpha_3^2=0,$};

					\node[] at (8,0.0) {$\alpha_1\beta=\alpha_2\beta=\alpha_3\beta=0.$};
					
					\draw [thick, <-] (adhm1) to [out=120,in=30,looseness=20] node[above] {$\alpha_1$}(adhm1);
					\draw [thick, <-] (adhm2) to [out=320,in=230,looseness=20] node[below] {$\alpha_3$}(adhm2);
					\draw [thick, <-] (adhm3) to [out=215,in=125,looseness=20] node[left] {$\alpha_2$}(adhm3);
					
					\draw [thick, <-] (adhm11)  -- node [above] {$\beta$} (adhm12);
				\end{tikzpicture}
				\qquad
				\begin{tikzpicture}[scale=1.0, every node/.style={scale=1.0}]
					
					\node[] at (3,1) {$1$};
					\node [above] (adhm11) at (3.1,0.9) {};
					\node [above] (adhm21) at (3.1,0.8) {};
					
					\node [above] (adhm1) at (3,1) {};
					\node [below] (adhm2) at (3,1) {};
					\node [below] (adhm3) at (3,1) {};
					
					%
					%
					
					%
					\node[] at (1.5,2.4) {$C:$};
					
					\node[] at (4.5,-1.5) {Figure 3};

					\node[] at (6,1.5) {$\alpha_1\alpha_2+\alpha_2\alpha_1=0,$};

					\node[] at (6,1.0) {$\alpha_1\alpha_3+\alpha_3\alpha_1=0,$};

					\node[] at (6,0.5) {$\alpha_2\alpha_3+\alpha_3\alpha_2=0,$};

					\node[] at (6,0) {$\alpha_1^2=\alpha_2^2=\alpha_3^2=0.$};
					
					\draw [thick, <-] (adhm1) to [out=120,in=30,looseness=20] node[above] {$\alpha_1$}(adhm1);
					\draw [thick, <-] (adhm2) to [out=320,in=230,looseness=20] node[below] {$\alpha_3$}(adhm2);
					\draw [thick, <-] (adhm3) to [out=215,in=125,looseness=20] node[left] {$\alpha_2$}(adhm3);
					
				\end{tikzpicture}
			\end{equation*}
			Let $S_{A}(i)$ be the simple $A$-module corresponding to the vertex $i$. Note that $A$ is the one-point coextension of $B$ by the simple $B$-module $S_{B}(1)$ and $B$ is the one-point extension of $C$ by the simple $C$-module $S_{C}(1)$.
			Also $\ITdist(C)=2$ (see Example \ref{ex-ex}). It follows from Corollary \ref{cor-point} that $\ITdist(A)=\ITdist(B)=\ITdist(C)=2$. Notice that $\Lambda$ is obtained from $A$ through $n$ times one-point extension and $m$ times one-point extension. Thus $\ITdist(\Lambda)=\ITdist(A)=2$ by Corollary \ref{cor-point}.
		}
	\end{Bsp}

\bigskip
\noindent{\bf Acknowledgements.}
Jinbi Zhang was supported by the National Natural Science Foundation of China (No. 12401038).

\noindent{\bf Declaration of interests.} The authors have no conflicts of interest to disclose.

\noindent{\bf Data availability.} No new data were created or analyzed in this study.

\end{document}